\documentclass[12pt]{amsart}
\allowdisplaybreaks
\usepackage[dvipsnames]{xcolor}
\usepackage{amsfonts,amsmath,latexsym,amssymb,verbatim,amsbsy,amsthm}
\usepackage{mathrsfs}
\usepackage{nicefrac}
\usepackage{graphicx}
\usepackage{nicefrac}
\usepackage{dsfont}
\usepackage{mathrsfs}
\usepackage{cancel}
\usepackage[normalem]{ulem}
\usepackage{color}
\usepackage{pifont}
\usepackage{tikz}
\usepackage{resizegather}
\usepackage{cancel}
\usetikzlibrary{arrows, positioning}

\usepackage{setspace} 
\usepackage{caption,subcaption} 

\usepackage{algorithm} 

\usepackage{todonotes}
\usepackage{algpseudocode} 

\usepackage{grffile}  

\usepackage[top=1in, bottom=1in, left=1in, right=1in]{geometry}

\usepackage{relsize}

\usepackage[colorlinks=true, pdfstartview=FitV, linkcolor=RoyalBlue,citecolor=ForestGreen, urlcolor=blue]{hyperref}
\usepackage[capitalize]{cleveref}
\numberwithin{equation}{section}
\numberwithin{figure}{section}
\numberwithin{table}{section}
\usepackage{float}

\renewcommand{\geq}{\geqslant}

\renewcommand{\leq}{\leqslant}
\renewcommand{\le}{\leqslant}
 
\newcommand{\be}{\begin{equation}}
\newcommand{\ee}{\end{equation}}

\theoremstyle{plain}
\newtheorem{THEOREM}{Theorem}[section]

\newtheorem{theorem}[THEOREM]{Theorem}
\newtheorem{corollary}[THEOREM]{Corollary}

\newtheorem{lemma}[THEOREM]{Lemma}
\newtheorem{proposition}[THEOREM]{Proposition}
\newtheorem{example}[THEOREM]{Example}
\newtheorem{assumption}[THEOREM]{Assumption}

\theoremstyle{definition}

\newtheorem{definition}[THEOREM]{Definition}

\theoremstyle{remark}
\theoremstyle{question}

\newtheorem{remark}[THEOREM]{Remark}

\newcommand{\E}{\mathbb{E}}

\newcommand{\R}{\mathbb{R}}
\newcommand{\der}[2]{\frac{d #1}{d #2}}
\newcommand{\norm}[1]{\left\|#1\right\|_{L^2,P}}
\newcommand{\myrangle}{\rangle_{L^2}}
\newcommand{\mylangle}{\langle}
\usepackage{color}
\usepackage{graphicx}
\usepackage{tikz-cd}
\usepackage{tikz}
\newcommand*{\horzbar}{\rule[.5ex]{2.5ex}{0.5pt}}

\DeclareMathOperator*{\argmin}{argmin}

\begin{document}

\title{Randomised Splitting Methods and Stochastic Gradient Descent}


\author{Luke Shaw}
\address{ironArray SLU, Castell\'{o} de la Plana, Spain}
\email{luke.shaw@ironarray.io}
\thanks{}

\author{Peter A. Whalley}
\address{Seminar for Statistics, Department of Mathematics, ETH Z{\"u}rich, Switzerland}
\email{peter.whalley@math.ethz.ch}
\thanks{}

\subjclass[2010]{Primary }

\date{}

\dedicatory{}

\begin{abstract}
We explore an explicit link between stochastic gradient descent using common batching strategies and splitting methods for ordinary differential equations. From this perspective, we introduce a new minibatching strategy (called Symmetric Minibatching Strategy) for stochastic gradient optimisation which shows greatly reduced stochastic gradient bias (from $\mathcal{O}(h^2)$ to $\mathcal{O}(h^4)$ in the optimiser stepsize $h$), when combined with momentum-based optimisers. We justify why momentum is needed to obtain the improved performance using the theory of backward analysis for splitting integrators and provide a detailed analytic computation of the stochastic gradient bias on a simple example.

Further, we provide improved convergence guarantees for this new minibatching strategy using Lyapunov techniques that show reduced stochastic gradient bias for a fixed stepsize (or learning rate) over the class of strongly-convex and smooth objective functions. Via the same techniques we also improve the known results for the Random Reshuffling strategy for stochastic gradient descent methods with momentum. We argue that this also leads to a faster convergence rate when considering a decreasing stepsize schedule. Both the reduced bias and efficacy of decreasing stepsizes are demonstrated numerically on several motivating examples.
\end{abstract}

\maketitle

\section{Introduction}
In machine learning, one common approach to training a model is to frame the task as an optimisation problem and seek the optimum of an objective function \( F: \mathbb{R}^d \to \mathbb{R} \) using an optimisation algorithm, in the particular case where $F$ corresponds to a finite-sum problem with gradient
\begin{equation}\label{eq:FiniteSum} \nabla F=\frac{1}{N}\sum_{i=1}^{N} \nabla f_i,\end{equation}
where $f_{i}:\mathbb{R}^{d}\to\mathbb{R}$ for $i = 1,...,N$.

In this setting the goal is to find the minimiser of $F$ (assuming there is a unique minimum),
\begin{equation}
X_*=\argmin_{X\in \mathbb{R}^{d}} F(X).
\end{equation}
It is well known that the gradient flow on $\mathbb{R}^{d}$, where $X(0) = x_{0}$,
\begin{equation}\label{eq:GD_cont}
    \frac{dX}{dt} = -\nabla F(X)\equiv G(X),
\end{equation}
converges to the minimiser $X_*$ as $t\to\infty$ under appropriate assumptions on $F$. The Gradient Descent (GD) algorithm, which, given a stepsize $h>0$ and an initial value $x_0\in \mathbb{R}^{d}$, follows the update rule
\begin{equation}\label{eq:GD}\tag{GD}
    x_{k+1}= x_k-h\nabla F(x_k),
\end{equation}
 can be seen as the Euler discretisation of the continuous dynamics \cref{eq:GD_cont}. Under appropriate assumptions on $h$ and $F$
\cref{eq:GD} also converges to the minimiser $X_* \in \mathbb{R}^{d}$ as $k\to \infty$ (see \cite{NesterovBook}). In particular, under the following assumptions on the objective function, $F$, one can easily achieve quantitative convergence guarantees.

\begin{definition}\label{def:Lsmooth}
A continuously differentiable $F:\mathbb{R}^{d} \to \mathbb{R}$ is $L$-smooth if  for all $x,y\in \mathbb{R}^{d}$  with $L>0$,
$$
\lVert \nabla F(x) - \nabla F(y)\rVert \le L \lVert x-y\rVert.
$$
\end{definition}
\begin{definition}\label{def:muconvex}
A continuously differentiable $F:\mathbb{R}^{d} \to \mathbb{R}$ is $\mu$-strongly convex if for all $x,y \in \mathbb{R}^{d}$ with $\mu>0$,
$$
F(x) \geq F(y)+\langle \nabla F(y),x-y\rangle + \frac{\mu}{2} \lVert x-y\rVert^2.
$$
Further, $F$ is (merely) convex if it is $0$-strongly convex.
\end{definition}

\begin{assumption}[$\mu$-strongly convex and $L$-smooth]\label{assum:smoothness}$F$ is continuously differentiable and there exists $L > \mu > 0$ such that $F$ is $\mu$-strongly convex and $L$-smooth.
\end{assumption}

\begin{theorem}\label{thm:GD}
    Let $F$ satisfy \cref{assum:smoothness} with minimiser $X_{*} \in \mathbb{R}^{d}$. For an initialisation $x_{0} \in \mathbb{R}^{d}$, stepsize $0<h<1/L$ and $K\in\mathbb{N}$, the iterates of \cref{eq:GD} satisfy
    \[
    \|x_{K}-X_{*}\|^{2}_{2} \leq (1-h\mu)^{K}\|x_{0}-X_{*}\|^{2}_{2}.
    \]
    \begin{proof}
        See \cite[Theorem 3.6]{Garrigos2023} for example.
    \end{proof}
\end{theorem}
In particular, for $h = 1/L$, gradient descent converges to the minimiser with a rate of $\mu/L$, which is known as the condition number of the objective $F$. 

Gradient descent is the simplest optimisation algorithm that one can use. Depending on the context, its limitations may be overcome in different ways, which we review in this introduction. If the gradient is too expensive to calculate, one may use a stochastic gradient descent algorithm, which we discuss in \cref{sec:StochGrad}; one may also introduce momentum, which has different benefits in the full versus stochastic gradient case, as discussed in \cref{sec:AccConv}; finally, we propose that the combination of stochastic gradient and momentum may be most effectively understood via the literature on \emph{splitting methods}, see \cref{sec:splitting_methods_intro}.

\subsection{Stochastic Gradients}\label{sec:StochGrad}
When performing optimisation within the statistics and machine learning context, where models increasingly rely on large datasets, evaluating the gradient $\nabla F$ requires computations over the entire dataset. It is thus computationally impractical to evaluate $\nabla F$ at every iteration. To address this, full-gradient methods have largely been replaced by optimisers which use an unbiased estimator $\nabla f_{\boldsymbol{\omega}}$ of $\nabla F$ (i.e $\mathbb{E}_{\boldsymbol{\omega}}[\nabla f_{\boldsymbol{\omega}}]=\nabla F,\forall x\in\R^d$ and where $\omega\sim\pi$ is a random variable), with reduced computational cost, called a \emph{stochastic gradient}. For example, replacing the full gradient in \cref{eq:GD} with a stochastic gradient gives Stochastic Gradient Descent (SGD) \cite{Robbins1951}

\begin{equation}\label{eq:SGD}\tag{SGD}
\begin{aligned}
x_{k+1}&= x_k-h{\nabla}f_{\boldsymbol{\omega}}(x_k), \quad \textnormal{where } \boldsymbol{\omega} \sim \pi.
\end{aligned}
\end{equation}

In fact, in machine learning applications, where problems are often non-convex, the use of a stochastic gradient may be desirable, since the added noise helps: a) find better local minima \cite{keskar2016large}; b) escape saddle points \cite{ge2015escaping}; and c) provide greater generalisation \cite{ge2015escaping,keskar2016large}.
 However, since the stochastic gradients are inexact, the algorithm—when using a fixed stepsize \( h \)—no longer asymptotically converges to the true minimiser $X_*$ meaning that $\lim_{K\to\infty}x_K$ is only approximately equal to $X_*$ (see \cite{Gower2019,Bottou2004}). Instead, it stabilises at a point some distance away, an effect known as the ``stochastic gradient bias'' \cite{Bach2011}. 
 Indeed, under appropriate assumptions, this bias may be quantified.

\begin{assumption}[Finite Variance of Stochastic Gradient]\label{assum:stochastic_gradient}
We assume that the variance of the stochastic gradient 
$\nabla f_{\boldsymbol{\omega}}$ at the true minimiser $X_{*} \in \mathbb{R}^{d}$ of $F$ is finite, i.e.,
$$
\sigma_*^2\equiv \E_{\boldsymbol{\omega} \sim \pi}\left\|\nabla f_{\boldsymbol{\omega}}(X_{*})\right\|^2<\infty.
$$
\end{assumption}

\begin{theorem}[SGD-RM]\label{thm:SGDRM}
Assume a function $F:\mathbb{R}^{d} \to \mathbb{R}$ satisfying \cref{assum:smoothness} with minimiser $X_{*} \in \mathbb{R}^{d}$ and consider iterates $(x_k)_{k\in \mathbb{N}}$ generated by SGD with stepsize $0<h<1/(2L)$ and starting iterate $x_0\in\R^d$. Assume the stochastic functions $ f_{\boldsymbol{\omega}}:\mathbb{R}^{d}\to\mathbb{R}$ are continuously differentiable, convex and $L$-smooth, $\forall \boldsymbol{\omega}\sim\pi$ and satisfy $\mathbb{E}_{\boldsymbol{\omega}}[\nabla f_{\boldsymbol{\omega}}]=\nabla F$ and \cref{assum:stochastic_gradient} with constant $0<\sigma^{2}_{*}<\infty$ then
$$
\E\|x_{K}-X_*\|^2\leq(1-h\mu)^{K}\|x_0-X_*\|^2+\underbrace{2h\frac{\sigma_*^2}{\mu}}_{\textnormal{stochastic gradient bias}}.
$$
\begin{proof}
       Follows from \cite[Proof of Thm. 3.1]{Gower2019} or \cite[Thm. 5.8]{Garrigos2023}.
\end{proof}
\end{theorem}

The stochastic gradient bias, present even when $K\to\infty$ is a distinct source of error to the finite contraction $(1-h\mu)^K$, which disappears in the asymptotic limit.
In order to remove the stochastic gradient bias and converge to $X_*$ one must use a decreasing stepsize schedule $(h_{k})_{k\in\mathbb{N}}$ satisfying the Robbins-Munro criterion: $\sum^{\infty}_{k=0}h_{k} = \infty$ and $\sum^{\infty}_{k=0}h^{2}_{k} < \infty$ \cite{Robbins1951}. The bias may also be reduced (but not eliminated entirely) via Polyak-Ruppert averaging, in which one averages the iterates after a burn-in phase and subsequently improves the accuracy of the approximation \cite{ruppert1988efficient,polyak1990new}. 

Another way to reduce the bias is to make a smart choice for the distribution $\pi$ from which one samples $\boldsymbol{\omega}$ in order to generate the stochastic gradient - that is, the ``randomisation strategy''. Implicitly, we have taken the simplest such strategy, known as Robbins-Monro (RM) and so the result of \cref{thm:SGDRM} may be said to hold for the SGD-RM algorithm. In fact, an alternative strategy, known as Random Reshuffling (RR), is  preferred in practice \cite{Recht2012,Sun2020,Bottou2009}. It is now well-established that RR offers not only greater computational efficiency — owing to improved caching and memory access \cite{Bengio2012} — but also a reduced stochastic gradient bias compared to RM when used in stochastic gradient descent (resulting in SGD-RR and SGD-RM, respectively) \cite{Mishchenko2020,Cha2023}. In particular, under the same assumptions on the objective and stochastic gradients as in \cref{thm:SGDRM}, one may show that the stochastic gradient bias is reduced from $\mathcal{O}(h)$ to $\mathcal{O}(h^2)$ (although with a more restrictive limit on the stepsize), and so, asymptotically, one obtains a better approximation of $X_*$ \cite{Mishchenko2020}.
 

\subsection{Momentum and Accelerated Convergence}\label{sec:AccConv}
While in \cref{thm:GD}, the only contribution to the error in estimating the true minimiser $X_*$ is due to finite contraction which disappears completely asymptotically, in \cref{thm:SGDRM}, the stochastic gradient bias introduces an additional error which is present even as the number of iterations $K\to\infty$.
With respect to contraction, in the full gradient case it is known that the $\mu/L$ convergence rate in \cref{thm:GD} may be accelerated by introducing momentum into the gradient descent method. Two popular choices are the Polyak Heavy Ball \cite{Polyak1964} and Nesterov methods \cite{NesterovBook},
\begin{align}
x_{k+1}&=x_k+\beta(x_k-x_{k-1})-\alpha \nabla F(x_k),\label{eq:Polyak_org}\tag{Polyak}\\
x_{k+1}&=x_k+\beta(x_k-x_{k-1})-\alpha \nabla F(x_k+\beta(x_k-x_{k-1})),\label{eq:Nesterov_org}\tag{Nesterov}
\end{align}
(for $x_0\equiv x_{-1}$ given) and $\alpha, \beta >0$. These are available as default choices of optimiser within the popular \texttt{pytorch} library \cite{paszke2019pytorch}. 
For an appropriate choice of parameters and certain function classes, Polyak and Nesterov accelerate convergence from the rate $\mu/L$ for \cref{eq:GD} in \cref{thm:GD} to $\sqrt{\mu/L}$. For the Polyak method, this acceleration is achieved only for quadratic objectives $F$ \cite{Polyak1964}; Nesterov's algorithm achieves the accelerated rate for all $\mu$-strongly convex and $L$-smooth $F$ \cite{NesterovBook}.

We shall not be concerned with the accelerated convergence of momentum-based methods in the full gradient context in this article. It is of greater interest to us how such methods behave when the full gradient $F$ in \cref{eq:Polyak_org,eq:Nesterov_org} is replaced by a stochastic gradient. In this way one may derive stochastic gradient momentum-based optimisers \cite{StochasticHeavyBall}. We show in this article that such stochastic gradient momentum-based methods in fact \emph{reduce the (asymptotic) stochastic gradient bias} contribution to the error in \cref{thm:SGDRM}. Results for such the RR strategy in combination with momentum-type methods have recently begun to appear in the literature \cite{tran2021smg}. Our work here follows this trend.

Interestingly, both \cref{eq:Polyak_org} and \cref{eq:Nesterov_org} may be seen as discretisations of a continuous system of equations (just as \cref{eq:GD} discretises \cref{eq:GD_cont}) on $\mathbb{R}^{2d}$
\begin{equation}\label{eq:Damped}
    \begin{split}
    \frac{dX}{dt} &= V\\
    \frac{dV}{dt} &= -\nabla F(X) - \gamma V=G(X)-\gamma V,
\end{split}
\end{equation}
where $\gamma > 0$ is known as a friction parameter. Note that taking $\gamma \to \infty$ and introducing a suitable time-rescaling one recovers \eqref{eq:GD} (see \cite[Sec 6.5]{pavliotis2014stochastic}).

For example, Polyak's heavy ball method may be rewritten, introducing the momentum variable $v_k$,
(with $v_0=0$ and $x_0$ given) as the update rule
\begin{equation}\label{eq:Polyak}\tag{HB}
\begin{split}
x_{k+1} &= x_{k} + h e^{-\gamma h}v_{k} - h^{2}\nabla F(x_{k})\\
v_{k+1} &= e^{-\gamma h} v_{k} - h \nabla F(x_{k}),
\end{split}    
\end{equation}
for a stepsize $h>0$ and initialisation $(x_{0},v_{0}) \in \mathbb{R}^{2d}$.
Nesterov's accelerated gradient descent, on the other hand, corresponds to an Additive Runge-Kutta (ARK) discretisation of \eqref{eq:Damped}
\begin{equation}\label{eq:Nesterov}\tag{NAG}
    \begin{split}
x_{k+1} &= x_{k} + h e^{-\gamma h}v_{k} - h^{2}\nabla F(x_{k} + he^{-\gamma h}v_{k})\\
v_{k+1} &= e^{-\gamma h}v_{k} - h \nabla F(x_{k} + h e^{-\gamma h}v_{k}),
\end{split}
\end{equation}
for a stepsize $h>0$ and initialisation $(x_{0},v_{0}) \in \mathbb{R}^{2d}$ \cite{Dobson2025}.

\subsection{Splitting Methods}\label{sec:splitting_methods_intro}
Given this connection between both standard and momentum-based optimisation methods and discretisations of continuous systems \cref{eq:GD_cont,eq:Damped}, it seems natural to turn to the numerical analysis literature on ordinary differential equations (ODEs). In particular, we study ``splitting methods", an important class of numerical integrators \cite{SplittingMethodsActa,SplittingMethodsODEsActa}. These are typically used when considering ODEs which have the following form
\begin{equation}\label{eq:ODE}
    \frac{dX}{dt} = G(X)= \sum^{N}_{i=1}g_{i}(x),
\end{equation}
which is clearly related to \cref{eq:FiniteSum,eq:GD_cont}.
For example, the celebrated ``leapfrog" method which is used to solve Hamiltonian dynamics and within Hamiltonian Monte Carlo may be understood as a splitting method \cite{duane1987hybrid,BouRabee2018}.
We will show in \cref{sec:splitting_methods} that one epoch of Random Reshuffling (RR) may be understood as a splitting method with a randomised splitting order. We refer to this type of splitting method as a ``randomised splitting method", a term used in the literature on quantum simulation \cite{zhang2012randomized,childs2019faster}. Moreover, inspired by the splitting method framework, we propose an alternative randomisation strategy, called \textit{symmetric minibatch splitting} (SMS) and justify its superiority using analysis techniques developed for splitting methods in \cref{sec:SplittingAnalysis}. An important consequence of viewing the stochastic gradient optimisers as splitting methods is that it becomes apparent why momentum-based methods obtain reduced stochastic gradient bias using SMS (or indeed RR), while plain SGD (without momentum) does not - a benefit that is distinct to the accelerated convergence obtained by Nesterov for full gradient optimisation as discussed in \cref{sec:AccConv}.
These connections between stochastic gradient methods and splitting methods have also been explored within the context of MCMC \cite{franzese2022revisiting,paulin2024sampling,shaw2025random}.

\subsection{Contributions}

\begin{itemize}

    \item We establish an explicit link between stochastic gradient optimisers (with and without momentum) and (randomised) splitting methods. In order to do this, we detail the batching strategies used to generate stochastic gradients in \cref{sec:randomisation_strategies}, and their relation to splitting methods in \cref{sec:splitting_methods}.
    \item Inspired by the link to splitting methods, we introduce a new randomisation strategy based on the celebrated Strang splitting \cite{Strang1963}, called \textit{symmetric minibatch splitting} (SMS). We show that, when combined with the Stochastic Heavy Ball method \cref{eq:Polyak}, it is provably more accurate, arriving at a convergence guarantee via Lyapunov function techniques.

    This randomisation strategy has recently been proposed in the machine learning literature under the moniker of ``flip-flop" \cite{rajput2022permutation}. However, combining SMS with \cref{eq:SGD} (as in \cite{rajput2022permutation}), giving SGD-SMS, does not yield improved guarantees beyond the quadratic case. By establishing the connection to splitting methods, we are able to explain this phenomenon and provide numerical methods which yield improved accuracy in a more general setting (beyond quadratics). By employing the Heavy Ball method \cref{eq:Polyak} with a stochastic gradient generated with SMS, we are able to double the order of accuracy compared to SGD-SMS, whilst only requiring one gradient per step. More precisely, we establish convergence guarantees where the stochastic gradient bias is $\mathcal{O}(h^4)$ in MSE compared to the state-of-the-art results of $\mathcal{O}(h^2)$ in \cite{Mishchenko2020}.
    
    \item In addition, we improve the known convergence results for Heavy Ball with random reshuffling ($\mathcal{O}(h^3)$ in MSE), which is provably more accurate then SGD-RR. We illustrate that these bounds are tight by constructing a model problem which yields these bounds.
    \item We empirically confirm the effects of different strategies on logistic regression problems.
\end{itemize}

\section{Randomisation Strategies}\label{sec:randomisation_strategies}
Following on from the introduction, an important consideration is how to generate the stochastic gradient estimator, $\nabla f_{\boldsymbol{\omega}}$, which is determined by the sampling of $\boldsymbol{\omega}$. We recall that we consider a class of problems where $\nabla F$ is based on a finite sum
 \begin{equation*}
 \nabla F=\frac{1}{N}\sum_{i=1}^{N}\nabla f_i.
  \end{equation*}
 The different randomisation strategies RM, RR and SMS may then be understood as special cases of the same algorithm \cref{alg:GenAlg}, where one generates a matrix $\Omega_{n,N}(m)\in\R^{m\times n}$
 \begin{equation}\label{eq:OmegaMatrix}
     \Omega_{n,N}(m)=\begin{bmatrix}
         \omega_{11} & \omega_{12} & \ldots &  \omega_{1n}\\
         \omega_{21} & \ddots & \ldots &  \omega_{2n}\\
         \vdots & \vdots & \ddots &  \vdots\\
        \omega_{m1} & \ldots & \ldots &  \omega_{mn}
     \end{bmatrix}
     =\begin{bmatrix}
          \horzbar & \boldsymbol{\omega}_{1} &\horzbar \\
         \horzbar & \boldsymbol{\omega}_{2} &\horzbar\\
          & \vdots &\\
        \horzbar & \boldsymbol{\omega}_{m} &\horzbar\\
     \end{bmatrix},
 \end{equation}
 of $mn$ distinct entries $\omega_{ij}$ drawn without replacement from $\{1,2,\ldots,N\}$. We will write simply $\Omega$ in the following to suppress the additional notation for brevity. 
 
 For the case $mn=N$, one has maximal $m=R=N/n$ (we assume $n$ divides $N$ exactly for simplicity), and the matrix exactly partitions the set of indices. A batch is then defined as the set of indices corresponding to an $\boldsymbol{\omega}_{i}=(\omega_{i1},\omega_{i2}, \ldots,\omega_{in})$ for each $i = 1,...,m$, and gives a stochastic gradient via
 \begin{equation}\label{eq:stochgrad}
    \nabla f_{\boldsymbol{\omega}_i}\equiv\frac{1}{n}\sum_{j=1}^{n}\nabla f_{\omega_{ij}}.
 \end{equation}
Note that no two batches in the same matrix then contain the same $\nabla f_i$. 

We may now define the aforementioned strategies for generating the stochastic gradient, RM, RR and SMS. Stochastic gradient algorithms are typically run for a total number of iterations which is an integer multiple of $R$,  $K=n_eR,n_e\in\mathbb{N}$. $n_e$ is then the \emph{number of epochs}, where an epoch is composed of $R$ iterations.

\subsection{Robbins-Monro (RM)} 
For RM, one sets $m=1$ and the stochastic gradient is then generated at each step $k \in \mathbb{N}$ via sampling $\Omega\in\mathbb{R}^{1\times n}$ and then calculating $\nabla f_{\boldsymbol{\omega}_{k}}$ via \cref{eq:stochgrad}.
In expectation then, after one epoch of $R$ iterations, the algorithm sees every $\nabla f_i$ for $i=1,...,N$ once \cite{Robbins1951}.

\subsection{Random Reshuffling (RR)}
In the case of RR, one sets $m=R$ first, samples $\Omega\in\R^{R\times n}$ and then calculates $\nabla f_{\boldsymbol{\omega}_{1}}$. Over the succeeding $R-1$ steps one iterates through the remaining batches $\boldsymbol{\omega}_2,\ldots,\boldsymbol{\omega}_R$, calculating the respective stochastic gradient approximation at each step before reaching the end of the epoch. One then resamples the matrix $\Omega$ and carries out another $R$ steps. Hence, after one epoch of $R$ iterations, the algorithm has seen every $\nabla f_i$ for $i=1,...,N$ once and only once.
For an intuition as to why RR is superior to RM, see \cite[Exercise 2.10]{BertsekasBook}.

\subsection{Symmetric Minibatch Splitting (SMS)}
In the case of SMS, similarly to RR one sets $m=R$ first, samples $\Omega\in\R^{R\times n}$ and then calculates $\nabla f_{\boldsymbol{\omega}_{1}}$. Over the succeeding $R-1$ steps one iterates through the remaining batches $\boldsymbol{\omega}_2,\ldots,\boldsymbol{\omega}_R$, before reaching the end of the epoch. One then reverses the order of the batches and carries out another $R$ steps. I.e. over the following epoch one iterates through the batches in the order $\boldsymbol{\omega}_R,\boldsymbol{\omega}_{R-1}\ldots,\boldsymbol{\omega}_1$. One then resamples the matrix $\Omega$ and carries out another $2R$ steps with the same procedure. Similarly, every epoch the algorithm has seen every $\nabla f_i$ for $i=1,...,N$ once and only once, however every other epoch the choice of batch is purely deterministic conditional on the ordering in the previous epoch.

\subsection{Stochastic Gradient Algorithm}
In \cref{alg:GenAlg} we provide a general stochastic gradient algorithm which details the minibatching procedure for SGD or momentum SGD (i.e. using Nesterov or Heavy Ball) if an even number of epochs are run. Note that this can also be used in combination with any gradient-based optimisation algorithm which relies on stochastic gradients, for example, more sophisticated optimisation algorithms like Adam or AdaGrad \cite{kingma2014adam,duchi2011adaptive}, which are not discussed in this work and for which theoretical guarantees are much more challenging to derive \cite{defossez2022simple}.
\begin{algorithm}[H]
    \textbf{Input}: $m,K,N,n,z_0,h$ \Comment{$m=1$ if RM, $m=R=N/n$ if RR or SMS}
    \caption{General Stochastic Gradient Algorithm}\label{alg:GenAlg}
    \begin{algorithmic}[1]
    \State $k\gets 0$
    \While{$k\leq2n_eR$} \Comment{Assume $K=2n_e R$, $n_e\in\mathbb{N}$}
        \State Sample $\Omega_{n,N}(m)=\{\boldsymbol{\omega}_i\}_{i=1}^m$ according to \cref{eq:OmegaMatrix} 
        \For{$i=1,\ldots,m$}
            \State Generate $z_{k+1}$ using $\nabla f_{\boldsymbol{\omega}_i}$ via \cref{eq:SGD}, \cref{eq:Polyak} or \cref{eq:Nesterov}
            \State $k\gets k+1$
        \EndFor
        \If {SMS} \Comment{Reverse the minibatch ordering for next epoch.}
        \For{$i=m,\ldots,1$}
            \State Generate $z_{k+1}$ using $\nabla f_{\boldsymbol{\omega}_{i}}$ via \cref{eq:SGD}, \cref{eq:Polyak} or \cref{eq:Nesterov}
            \State $k\gets k+1$
        \EndFor
        \EndIf
    \EndWhile
    \State \Return{$(z_k)_{k=0}^K$}
    \end{algorithmic}
\end{algorithm}

\section{Splitting Methods}\label{sec:splitting_methods}
Recall that we are interested in integrating systems of the form
\begin{equation}\label{eq:EpochFlow}
\der{X}{t}\equiv\dot{X}=G(X)=\frac{1}{R}\sum_{i=1}^R g_{\boldsymbol{\omega}_i}(X),
\end{equation}
where $G(X)=-\nabla F(X),g_{\boldsymbol{\omega}_i}(X)=-\nabla f_{\boldsymbol{\omega}_i}(X)$,
over a time $Rh$, where we do not have access to the exact solution $\phi_h^{[GF]}$ for the gradient flow.
We could then approximately solve the flow using the simple Euler method, $x\gets x+RhG(x)$, which gives the gradient descent algorithm \cref{eq:GD} with timestep $Rh$.  As is clear from \cref{thm:GD}, this is not feasible since it would require that $h=\mathcal{O}((RL)^{-1})$ for convergence. Alternatively, one may have recourse to a \emph{splitting method}, since the flow has a natural decomposition as the sum of subproblems with  $g_{\boldsymbol{\omega}_i}$ \cite{SplittingMethodsODEsActa}. One may then use the gradient descent method to (approximately) solve each subproblem  via $\psi_{h,i}^{[GD]}(x)=x+(Rh)\frac{g_{\boldsymbol{\omega}_i}(x)}{R}=x+hg_{\boldsymbol{\omega}_i}(x)$, giving a composition integrator \emph{with the same cost} as the basic Euler method which uses the full gradient $G$ with  timestep $Rh$
\begin{equation}\label{eq:EulerComp}
\Psi_{Rh}(x)=\psi_{h,R}^{[GD]}\circ\psi_{h,R-1}^{[GD]}\circ\ldots\circ \psi_{h,1}^{[GD]}(x)\approx x+RhG(x).
\end{equation}
A Taylor expansion shows that the composition approximates the exact full-gradient flow over a timestep $Rh$ up to $\mathcal{O}(h^2)$, and so one could hope that it also converges close to the minimiser as does the full gradient gradient descent scheme. However, assuming that each flow $g_{\boldsymbol{\omega}_i}(x)\approx G(x)$ (i.e., the Lipschitz constants are similar and all $g_i$ are strongly convex), one has a more reasonable timestep restriction for contraction of $h=\mathcal{O}(L^{-1})$ \cite[Thm. 1]{Mishchenko2020}. If one does not randomise the batching to generate $\{g_{\boldsymbol{\omega}_i}\}_{i=1}^R$, this is known as the Incremental Gradient (IG) method; if one randomly batches once prior to executing the descent algorithm, this is known as the Shuffle Once (SO) method. Alternatively, one may re-batch the $\{g_{\boldsymbol{\omega}_i}\}_{i=1}^R$ after one epoch (of duration $Rh$ in `time') and use a new flow \cref{eq:EpochFlow} $\{g_i\}_{i=1}^R$, and apply the corresponding composition \cref{eq:EulerComp}. This gives the Random Reshuffle (RR) method, which has the advantage of avoiding using `bad' batchings many times, which can occur with IG or SO. Finally, one may also consider using a \emph{symmetric} composition integrator (using multiplication to denote composition)
\begin{equation}\label{eq:SymEulerComp}
\Psi_{2Rh}(x)=\prod_{i=R}^1\psi_{h,i}^{[GD]}\circ\prod_{i=1}^R\psi_{h,i}^{[GD]}(x)\approx x+2RhG(x),
\end{equation}
with timestep $2Rh$, of equivalent cost to \cref{eq:EulerComp}. One might hope that the symmetric scheme, due to error cancellation, in general converges to a point even closer to the minimiser than \cref{eq:EulerComp}, but in fact this is only the case for linear $G,g_{\boldsymbol{\omega}_i}$, (cf. \cref{sec:ModelProblem}). In general, the composition \cref{eq:SymEulerComp} approximates the exact full-gradient flow over a timestep $Rh$ up to $\mathcal{O}(h^2)$, the same order as for \cref{eq:EulerComp}. Naturally, one also expects to be able to use a timestep $h=\mathcal{O}(L^{-1})$. One can also imagine Symmetric IG and Symmetric SO methods; we focus on the fully randomised version which we call Symmetric Minibatch Splitting (SMS).

\subsection{Second-order dynamics}
One may also write the second-order dynamics \cref{eq:Damped} based on the gradient $G$ as 
\begin{equation}\label{eq:SecondOrderFullGrad}
    \dot{Z}=
    \begin{bmatrix}\dot{X}\\\dot{V}\end{bmatrix}=\begin{bmatrix}V\\G(X)-\gamma V\end{bmatrix}=\begin{bmatrix}V\\0\end{bmatrix}+\begin{bmatrix}0\\G(X)-\gamma V\end{bmatrix}=AZ+BZ,
\end{equation}
where ($\rho:\R^d\to\R$)
\begin{equation}\label{eq:Operators}
    A\rho(Z)=V\nabla_X\rho(Z),B\rho(Z)=(G-\gamma V)\nabla_V\rho(Z)
\end{equation}
are operators which act on vector fields and return vector fields. As for \cref{eq:EpochFlow} one does not have access to the full solution flow $\phi_{h}^{[A+B]}$, but instead of using Euler as the basic map, one may construct a splitting integrator as follows. This formulation is especially advantageous since one may write
\begin{equation}\label{eq:SecondOrderEpochFlow}
    \dot{Z}=\frac{1}{R}\sum_{i=1}^R\begin{bmatrix}V\\g_{\boldsymbol{\omega}_i}(X)-\gamma V\end{bmatrix}=\frac{1}{R}\sum_{i=1}^R\begin{bmatrix}V\\0\end{bmatrix}+\frac{1}{R}\sum_{i=1}^R\begin{bmatrix}0\\g_{\boldsymbol{\omega}_i}(X)-\gamma V\end{bmatrix}=\frac{1}{R}\sum_{i=1}^RAZ+B_{\boldsymbol{\omega}_i}Z,
\end{equation}
where the subproblems $\dot{Z}=AZ=(V,0)$ and $\dot{Z}=B_{\boldsymbol{\omega}_i}Z=(0,g_{\boldsymbol{\omega}_i}-\gamma V)$ admit exact solutions
\begin{equation}\label{eq:ExactABMaps}
\phi_{h}^{[A]}(Z)=(X+hV,V),\quad \phi_h^{[B_{\boldsymbol{\omega}_i}]}(Z)=(X,e^{-\gamma h}V+hg_{\boldsymbol{\omega}_i}(X)).
\end{equation}
To implement each stochastic gradient iteration, one may use a basic composition in either a non-symmetric form 
$\phi_{h}^{[A]}\circ\phi_{h}^{[B_{\boldsymbol{\omega}_i}]}$
(called Lie-Trotter), or a symmetric form $ \phi_{h/2}^{[A]}\circ\phi_{h}^{[B_{\boldsymbol{\omega}_i}]}\circ\phi_{h/2}^{[A]}$ (called Strang). One sees from the forms of the maps \cref{eq:ExactABMaps} that the Lie-Trotter composition is equivalent to Heavy Ball:
$$
\begin{aligned}
&v_{k+1} = e^{-\gamma h} v_{k} + hg_{\boldsymbol{\omega}_i}(x_{k})\\
&x_{k+1}= x_{k} + hv_{k+1}=x_k+he^{-\gamma h}v_{k} +h^{2}g_{\boldsymbol{\omega}_i}(x_{k}).
\end{aligned}
$$
 The full map for an epoch may then be written as a composition integrator based on the non-symmetric Lie-Trotter-like splitting,
\begin{equation}\label{eq:ABComp}
\Psi_{Rh}=\phi_{h}^{[B_{\boldsymbol{\omega}_R}]}\circ\phi_{h}^{[A]}\circ\phi_{h}^{[B_{\boldsymbol{\omega}_{R-1}}]}\circ\ldots\phi_{h}^{[B_{\boldsymbol{\omega}_1}]}\circ \phi_{h}^{[A]}\approx \phi_{h}^{[A+B]},
\end{equation}
or an alternative splitting
\begin{equation}\label{eq:SymABComp}
\begin{aligned}
\Psi_{2Rh}&=\prod_{i=R}^1\left(\phi_{h}^{[B_{\boldsymbol{\omega}_i}]}\circ \phi_{h}^{[A]}\right)\circ\prod_{i=1}^R\left(\phi_{h}^{[B_{\boldsymbol{\omega}_i}]}\circ \phi_{h}^{[A]}\right)\\
&=\phi_{-h/2}^{[A]}\circ\underbrace{\prod_{i=R}^1\left(\phi_{h/2}^{[A]}\circ\phi_{h}^{[B_{\boldsymbol{\omega}_i}]}\circ\phi_{h/2}^{[A]}\right)\circ\prod_{i=1}^R\left(\phi_{h/2}^{[A]}\circ\phi_{h}^{[B_{\boldsymbol{\omega}_{i}}]}\circ \phi_{h/2}^{[A]}\right)}_{\text{Symmetric Strang composition}}\circ\phi_{h/2}^{[A]}
\end{aligned}
\end{equation}
which one can see is almost a symmetric composition of symmetric Strang compositions. Indeed this only becomes more clear by examining the flow over $2K$ epochs
$$
\left(\Psi_{2Rh}\right)^{K}=\phi_{-h/2}^{[A]}\circ\left(\prod_{i=R}^1\left(\phi_{h/2}^{[A]}\circ\phi_{h}^{[B_{\boldsymbol{\omega}_i}]}\circ\phi_{h/2}^{[A]}\right)\circ\prod_{i=1}^R\left(\phi_{h/2}^{[A]}\circ\phi_{h}^{[B_{\boldsymbol{\omega}_i}]}\circ\phi_{h/2}^{[A]}\right)\right)^K\circ\phi_{h/2}^{[A]}.
$$
This is important since it is possible to show that the inner symmetric composition in \cref{eq:SymABComp} approximates the full flow $\phi_{Rh}^{[A+B]}$ up to $\mathcal{O}(h^3)$, while \cref{eq:ABComp} is only $\mathcal{O}(h^2)$ (the same as \cref{eq:EulerComp,eq:SymEulerComp}). This a consequence of the use of maps $\phi_h^{[A]},\phi_{h}^{[B_{\boldsymbol{\omega}_i}]}$ which are exact solutions to an ODE, and thus obey $\Psi_{-h}^{-1}=\Psi_h$ (unlike the basic Euler map). In general then, the symmetric Strang composition scheme will converge to a point closer to the minimiser (to which the full flow converges) than the other schemes. Since the full scheme in \cref{eq:SymABComp} is \emph{conjugate} to the inner scheme, (related by simple changes of coordinates at the beginning and end of the algorithm) it can be expected to have the same long-time behaviour\footnote{This may be seen informally since, if one initialises with $v_0=0$, the map $\phi_{h/2}^{[A]}$ does nothing, and then since $v\approx0$ at the end of optimisation, the final map $\phi_{-h/2}^{[A]}$ also does not greatly affect the final $x$ iterate.} \cite{Casas2023}. This is indeed the case, as shown by the proof of \cref{theorem:SGPolyak} for standard (Lie-Trotter) Heavy Ball. Understanding stochastic optimisers as splitting methods makes it obvious that momentum is key to obtaining the benefit of SMS for general functions (not just quadratic functions, as was shown for \cref{eq:SymEulerComp} in \cite{rajput2022permutation}).

\section{Backward Error Analysis of Splitting Methods}\label{sec:SplittingAnalysis}
In order to analyse splitting methods, several tools have been developed. We now review these tools and explicitly relate them to the optimisation context. The first of these concepts is the Lie derivative.

\paragraph{\textbf{Lie derivative}}
Consider the gradient flow $G=-\nabla F$, 
$G:\mathbb{R}^d\mapsto\mathbb{R}^d$ which has flow $\phi_h$ solving $dX/dt=G(X)$, i.e. $X(t+h)=\phi_h(X(t))$. Then the Lie derivative $\mathcal{L}_G$ of a general function $\rho:\mathbb{R}^d\mapsto \mathbb{R} $ is defined by
\begin{equation}
\mathcal{L}_G\rho(x)=\sum_{i=1}^dG^{(i)}(x)\partial_{X^{(i)}}\rho(x)=\langle\nabla \rho(x), G(x)\rangle.\label{eq:LieDerivative}
\end{equation}
Consequently $\left.(d/dt)\right|_{t=0}\rho(\phi_t(x))=\mathcal{L}_G\rho(x)$. Note also that the Lie derivative is linear in its vector field (i.e. its subscript), since for $f,g:\R^{d}\mapsto\R^{d}$,
\begin{equation}\label{eq:LieLin}
\mathcal{L}_f+\mathcal{L}_g=\mathcal{L}_{f+g}.
\end{equation}

Such relations extend immediately to elementwise operations on functions $\rho:\mathbb{R}^d\to\R^m$ (see \cite[Sec. 9.1]{BouRabee2018}).

The Lie derivative can be used to Taylor expand $\rho$ along the flow $\phi_h$ of $G$ about a point $x$ using 
$
\rho(\phi_h(x))=\rho(x)+\sum_{n\geq1}\left.\der{^n}{h^n}\right|_{h=0}\rho(x),
$
where it may be shown that the iterated operator $\left.\der{^n}{h^n}\right|_{h=0}=\mathcal{L}_G^n$ and consequently
\cite[Sec III.5.1]{HairerBook}
\begin{equation}\label{eq:ExpansionFlow}
\rho(\phi_h(x))=\rho(x)+\sum_{n\geq1}\frac{h^n}{n!}\mathcal{L}_G^n\rho(x)=e^{h\mathcal{L}_G}\rho(x),
\end{equation}
a relation which can be shown to hold rigourously \cite{Olver1993Book} (see \cref{def:ExpLog}).
Such relations extend immediately to elementwise operations and thus by setting $\rho=id$, one recovers that the solution can be written in the form
$\phi_h(x)=e^{h\mathcal{L}_G}x$.

\paragraph{\textbf{Taylor Expansion of Optimiser}} One then considers the mapping generated by an optimiser step $\Psi_h$  as \cite{BCM2008}  
\begin{equation}\label{eq:ExpansionOptimiser}
\rho(\Psi_h(x))=\rho(x)+\sum_{n\geq1}\frac{h^n}{n!}\left.\der{^n}{h^n}\right|_{h=0}\rho(\Psi_h(x))=\left(I+\sum_{n\geq1}h^nC_n\right)\rho(x)=C(h)\rho(x),
\end{equation}
so that (taking $\rho=id$) one sees that the expansion \cref{eq:ExpansionOptimiser} matches that of the flow \cref{eq:ExpansionFlow} up to order $p$ if $C_n=\mathcal{L}_G/n!,1\leq n\leq p$.

Similarly one may identify $D(h)=\log(C(h)),C(h)=\exp(D(h))$ and thus expand \cite{BlanesBook}
\begin{equation*}
D(h)=\sum_{n\geq 1}h^n\sum_{m=1}^n\frac{(-1)^{m+1}}{m}\sum_{j_1+\ldots+j_m=n}C_{j_1}\ldots C_{j_m}\equiv\sum_{n\geq 1}h^n D_n,
\end{equation*}
from which one derives the order conditions $D_1=\mathcal{L}_G,D_n=0,2\leq n\leq p$ via $\rho(\Psi_h(x))=\exp(D(h))\rho(x)$. Note that for the adjoint method $\Psi^*_h(x)=\Psi^{-1}_{-h}(x)$, $\rho(\Psi^*_h(x))=\exp(-D(-h))\rho(x)$. Hence for a symmetric method (for which $\Psi^*_h(x)=\Psi_h(x)$) $D(-h)=-D(h)$ and thus $D_{2m}=0,m=1,2,3\ldots$.

\paragraph{\textbf{Modified Gradient Flow}} One may (formally) associate $D(h)=h\mathcal{L}_{\widetilde{G}_h}$ to a \emph{modified gradient flow} (modified equation in numerical analysis terms)
\begin{equation}\label{eq:ModifiedEquation}
    \widetilde{G}_h=G+hG_1+h^2G_2+\ldots,
\end{equation}
such that $D_n=h^{n+1}\mathcal{L}_{G_n}$ via the linearity of the Lie derivative \cref{eq:LieLin}. Note that then the optimiser converges to the true minimiser $X_*$ if and only if $X_*$ is also a fixed point of $\widetilde{G}_h$. Thus rather than analysing the error in the iterates $\|x_k-X_*\|$ for $k \in \mathbb{N}$ directly, we may perform \emph{backward error analysis} using the modified equation, which allows one to study the long-time behaviour (asymptotic convergence and errors) of optimisation methods. We shall demonstrate the procedure for the case of gradient descent \cref{eq:GD}.

\begin{example}[Gradient Descent]
    For gradient descent, one has $\Psi_h(x)=x+hG(x)$, and so the expansion \cref{eq:ExpansionOptimiser} follows
    \begin{equation*}\begin{aligned}            \rho(\Psi_h(x))&=\rho(x)+h\langle \rho'(x),G(x)\rangle+\frac{h^2}{2}\langle G(x),\rho''(x)G(x)\rangle+\mathcal{O}(h^3)\\
    &=\rho(x)+h\mathcal{L}_G\rho(x)+\frac{h^2}{2}\left(\mathcal{L}_G^2\rho(x)-\langle \rho'(x), G'(x)G(x)\rangle\right)+\mathcal{O}(h^3),
    \end{aligned}
    \end{equation*}
    giving $C(h)=I+h\mathcal{L}_G+\frac{h^2}{2}(\mathcal{L}_G^2-\mathcal{L}_{G'G'})+\mathcal{O}(h^3)$, where $\rho':\R^d\to\R^d$ is a vector and $G',\rho'':\R^d\to\R^{d\times d}$ are matrices. Using the definition of the logarithm of an operator \cref{def:ExpLog}, one has that
    $D(h)=\log(C(h))=h\mathcal{L}_G-\frac{h^2}{2}\mathcal{L}_{G'G}+\mathcal{O}(h^3)$.
    Thus the modified gradient flow of gradient descent is
    \begin{equation}\label{eq:GDModEq}
    \widetilde{G}_h=G-\frac{h}{2}G'G+\mathcal{O}(h^3).
    \end{equation}
    Since $G(X_*)=0$, $\widetilde{G}_h(X_*)=\mathcal{O}(h^3)$. In fact, since the general term of \cref{eq:ExpansionOptimiser} follows
    $$
    \der{^n}{h^n}\rho(X_*+hG(X_*))=\der{^{n-1}}{h^{n-1}}\langle \rho'(X_*+hG(X_*)),G(X_*)\rangle=\left\langle \der{^{n-1}}{h^{n-1}}\rho'(X_*+hG(X_*)),G(X_*)\right\rangle=0,
    $$
    necessarily one has that $\widetilde{G}_h(X_*)=0$, i.e. gradient descent converges to the true minimiser. 
\end{example}

\paragraph{\textbf{Baker-Campbell-Hausdorff}} The exponential formalism is useful since compositions of maps may be written as products of exponentials\footnote{Note the reversed order of the operators relative to the applied steps \cite{SplittingMethodsODEsActa}.},
$$
\phi^{[f]}_t\circ\phi^{[g]}_t(x)=e^{t\mathcal{L}_g}e^{t\mathcal{L}_f}=\exp\left(t\mathcal{L}_g+t\mathcal{L}_f+\frac{t^2}{2}[\mathcal{L}_g,\mathcal{L}_f]+\ldots\right),
$$
which via the Baker-Campbell-Hausdorff (BCH) formula may be expanded in iterated commutators of the operators $[A,B]=AB-BA$ (\cite{Varadarajan2013,Postnikov1986},\cite[Sec. III.4]{HairerBook}). See also \cite[Appendix A.4]{BlanesBook}.

The commutator of Lie derivatives may be related directly to a vector field in the following way. 
Consider two vector fields $f,g:\R^{d}\mapsto\R^{d}$. One can see that the commutator $[\cdot,\cdot]$ applied to a test function $\rho:\R^d\mapsto\R$ is 
$$
[\mathcal{L}_f,\mathcal{L}_g]\rho=\rho'(g'f-f'g)=\mathcal{L}_{(f,g)}\rho
$$
where we define the Lie bracket of vector fields $(f,g)=g'f-f'g$, with $g',f'$ the (square) Jacobian matrices of the vector valued functions. All this may be easily extended to elementwise operations for $\rho:\R^d\mapsto\R^m$. Note that the result $(f,g):\R^d\mapsto\R^d $, just as the commutator of two Lie derivatives returns a Lie derivative.

\subsection{Asymptotic Convergence Error in Optimisation} In order to use modified equations to examine stochastic gradient optimisation, we must understand how the local errors at each step propagate\footnote{Note that we will understand `step' as either a single iteration (for RM), an epoch (for RR), or 2 epochs (for SMS).}.
Since the asymptotic error (or \emph{bias}) is what concerns us, we may pick any starting point for the iterates, and thus we select $x_0=X_*$, which facilitates the analysis.
In the following, we define the norm $\norm{\cdot}^2=\mathbb{E}_{\boldsymbol{\omega}}[\|\cdot\|_{P}^2]$ with $P$ a matrix appropriately defined so that the optimiser is contractive, although we shall gloss over these details for brevity (see \cite[Sections 4.3 and 5]{sanz2021wasserstein} for a full treatment\footnote{For quadratic objectives \eqref{eq:SecondOrderFullGrad} is contractive in a matrix norm for a particular choice of $P$ for all $\gamma > 0$ (see for example \cite{Monmarche2023}). For $L$-smooth, strongly convex objectives the results of \cite{sanz2021wasserstein} only hold in the overdamped regime $\gamma \geq \mathcal{O}(\sqrt{L})$. In the low-friction regime we later use Lyapunov techniques to establish rigorous bounds for the Heavy Ball method.}).
\begin{assumption}\label{ass:InnerProd}
Let $\alpha_{h,\boldsymbol{\omega}}(x)+\beta_{h,\boldsymbol{\omega}}(x)=\Psi_h(x)-\phi_h(x)$ be a decomposition of the one-step error of the optimisation algorithm such that 
$$
\norm{\alpha_{h,\boldsymbol{\omega}}(X_*)}=\mathcal{O}(h^{(p+1)/2}),\quad \norm{\beta_{h,\boldsymbol{\omega}}(X_*)}=\mathcal{O}(h^{(p+2)/2}),
$$ and with $e_K=\Psi_h^{K}(X_*)-X_*$, $K \in \mathbb{N}$
\begin{equation}\label{eq:Assumption1a}
    \left|\mathbb{E}\mylangle \alpha_{h,\boldsymbol{\omega}}(X_*),\Psi_h(\Psi_h^{K}(X_*))-\Psi_h(X_*)\myrangle\right|\leq C_0h\norm{\alpha_{h,\boldsymbol{\omega}}(X_*)} \norm{e_{K}}.
\end{equation}
In addition, $\Psi_h$ is a contractive mapping (for appropriate choices of problem parameters and matrix $P$)
$$
\norm{\Psi_h(x)-\Psi_h(y)}^2\leq (1-C_1h)\norm{x-y}^2,\quad \forall x,y,\quad C_1>0.
$$
\end{assumption}

Under these assumptions, it can be shown (see \cref{sec:AsymptoticConvergenceError}) that, if one has $\norm{\alpha_{h,\boldsymbol{\omega}}(X_*)}=\mathcal{O}(h^{(p+1)/2})$, $\norm{\beta_{h,\boldsymbol{\omega}}(X_*)}=\mathcal{O}(h^{(p+2)/2})$, one expects the global asymptotic convergence error $\lim_{K\to\infty}\norm{e_K}=\mathcal{O}(h^{p/2})$ (i.e. a MSE of $\mathcal{O}(h^{p})$).

We now undertake a backward error analysis of the SGD algorithms.
\begin{example}[SGD-RM]
    For stochastic gradient descent with Robbins-Monro sampling of the gradient, one has that one optimisation step follows $\Psi_h(x)=x+hg_{\boldsymbol{\omega}}(x)$, and thus (cf. \cref{eq:GDModEq}) each step follows a (random) modified flow
    \begin{equation}\label{eq:SGDModEq}
    \widetilde{g}_{h,\boldsymbol{\omega}}=g_{\boldsymbol{\omega}}-\frac{h}{2}g_{\boldsymbol{\omega}}'g_{\boldsymbol{\omega}}+\mathcal{O}(h^3).
    \end{equation}

    Hence $X_*$ is not a fixed point of $\widetilde{g}_{h,\boldsymbol{\omega}}$. We may write one step of SGD via the exponential form $\exp(h\mathcal{L}_{\widetilde{g}_{h,\boldsymbol{\omega}}})$. Consequently, the local error admits the decomposition
    $\left(e^{h\mathcal{L}_{\widetilde{g}_{h,\boldsymbol{\omega}}}}-e^{h\mathcal{L}_G}\right)X_*=\alpha_{h,\boldsymbol{\omega}}(X_*)+\beta_{h,\boldsymbol{\omega}}(X_*)$, where $\alpha_{h,\boldsymbol{\omega}}(X_*)=h\left(\mathcal{L}_{g_{\boldsymbol{\omega}}}-\mathcal{L}_G\right)X_*$ is $\mathcal{O}(h)$ and $\norm{\beta_{h,\boldsymbol{\omega}}(X_*)}=\mathcal{O}(h^2)$. One may verify that $\alpha_{h,\boldsymbol{\omega}}(X_*)$ fulfills \cref{ass:InnerProd} (see \cref{eq:SGDRMAssump}) with
    $$
\norm{\alpha_{h,\boldsymbol{\omega}}(X_*)}=h\norm{\left(\mathcal{L}_{{g}_{\boldsymbol{\omega}}}-\mathcal{L}_G\right)X_*}
=h\norm{g_{\boldsymbol{\omega}}(X_*)-G(X_*)}=h\sigma_*,
$$
    and so the asymptotic error is $\mathcal{O}(h^{1/2})$ as is well known \cite[Thm. 3.1]{Gower2019}.
\end{example}
\begin{remark}
For SGD-RM, one has $\alpha_{h,\boldsymbol{\omega}}(X_*)=h g_{\boldsymbol{\omega}}(X_*)$, which fulfills \cref{eq:Assumption1a}
since (via the Cauchy-Schwartz inequality)
\begin{equation}\label{eq:SGDRMAssump}
    \begin{aligned}
&\mathbb{E}\langle -h\nabla g_{\boldsymbol{\omega}}(X_*),\Psi_h^{K}(X_*)-X_*-h[\nabla g_{\boldsymbol{\omega}}(\Psi_h^K(X_*))-\nabla g_{\boldsymbol{\omega}}(X_*)]\myrangle\\
&=h\mathbb{E}\langle \alpha_{h,\boldsymbol{\omega}}(X_*),-[\nabla g_{\boldsymbol{\omega}}(\Psi_h^K(X_*))-\nabla g_{\boldsymbol{\omega}}(X_*)]\myrangle\leq hL\norm{\alpha_{h,\boldsymbol{\omega}}(X_*)}\norm{e_K},
\end{aligned}
\end{equation}
using that $\mathbb{E}\langle \alpha_{h,\boldsymbol{\omega}}(X_*), e_K\myrangle=0$.
\end{remark}

\begin{example}[SGD-RR]
    For stochastic gradient descent with Random Reshuffle sampling of the gradient, one has to perform the analysis at the level of an epoch. One optimisation `step' is thus composed of $R$ iterations of the map $x\to x+hg_{\boldsymbol{\omega}}(x)$, each with a modified equation of the form \cref{eq:SGDModEq} for the batches defined by $\{\boldsymbol{\omega}_{i}\}_{i=1}^R$, and so one has $\Psi_{Rh}(x)=\prod_{i=1}^R\exp(h\mathcal{L}_{\widetilde{g}_{h,\boldsymbol{\omega}_i}})x$.
    One may use the BCH formula to derive that
    \begin{equation}
        \begin{aligned}
    \Psi_{Rh}&=\exp\left(h\sum_{i=1}^R\mathcal{L}_{\widetilde{g}_{h,\boldsymbol{\omega}_i}}+\frac{h^2}{2}\sum_{i=1}^R\left[\mathcal{L}_{\widetilde{g}_{h,\boldsymbol{\omega}_i}},\sum_{j=i+1}^{R}\mathcal{L}_{\widetilde{g}_{h,\boldsymbol{\omega}_j}}\right]+\mathcal{O}(h^3)\right)\\
    &=\exp\left(Rh\mathcal{L}_G+\frac{h^2}{2}\left(\sum_{i=1}^R\left[\mathcal{L}_{g_{\boldsymbol{\omega}_i}},\sum_{j=i+1}^{R}\mathcal{L}_{g_{\boldsymbol{\omega}_j}}\right]-\mathcal{L}_{g'_{\boldsymbol{\omega}_i}g_{\boldsymbol{\omega}_i}}\right)+\mathcal{O}(h^3)\right).
        \end{aligned}
    \end{equation}
    Hence  
    \begin{equation}\label{eq:SGDRRModEq}
    \widetilde{G}_{Rh,\Omega}=RG+\frac{h}{2}\sum_{i=1}^Rg'_{\boldsymbol{\omega}_i}\left(\sum_{j=1}^{i-1}g_{\boldsymbol{\omega}_j}-\sum_{j=i}^Rg_{\boldsymbol{\omega}_j}\right)+\mathcal{O}(h^2),
    \end{equation}
    where $\widetilde{G}_{Rh,\Omega}$ is the (random) modified gradient flow of SGD-RR (over an epoch) for the batch ordering determined by the random matrix $\Omega$ defined in \cref{eq:OmegaMatrix}. Again, $X_*$ is not a fixed point of $\widetilde{G}_{Rh,\Omega}$. However, since the leading order error $$\frac{h}{2}\sum_{i=1}^Rg'_{\boldsymbol{\omega}_i}\left(\sum_{j=1}^{i-1}g_{\boldsymbol{\omega}_j}-\sum_{j=i}^Rg_{\boldsymbol{\omega}_j}\right),$$ does not fulfill \cref{ass:InnerProd} (since it does not have expectation zero conditioned on $e_K$, as was the case for SGD-RM in \cref{eq:Assumption1a}), we set $\alpha_{h,\boldsymbol{\omega}}(x)=0$, and then, writing one step of SGD-RR as $\exp(h\mathcal{L}_{\widetilde{G}_{h,\Omega}})$,the local error follows
    \begin{equation}
    \begin{aligned}
\norm{\beta_{h,\boldsymbol{\omega}}(X_*)}&=\frac{h^2}{2}\norm{\left(\sum_{i=1}^R\left[\mathcal{L}_{g_{\boldsymbol{\omega}_i}},\sum_{j=i+1}^{R}\mathcal{L}_{g_{\boldsymbol{\omega}_j}}\right]-\mathcal{L}_{g'_{\boldsymbol{\omega}_i}g_{\boldsymbol{\omega}_i}}+\mathcal{O}(h)\right)X_*}\\
&=\frac{h^2}{2}\norm{\sum_{i=1}^Rg_{\boldsymbol{\omega}_i}'(X_*)\left(\sum_{j=1}^{i-1}g_{\boldsymbol{\omega}_j}(X_*)-\sum_{j=i}^Rg_{\boldsymbol{\omega}_j}(X_*)\right)+\mathcal{O}(h)}=\mathcal{O}(h^2),
    \end{aligned}
    \end{equation}
    and so (since $p=2$) the asymptotic error one expects to see is $\mathcal{O}(h)$, as is indeed the case \cite[Thm. 1]{Mishchenko2020}.
\end{example}

\begin{example}[SGD-SMS]
    For stochastic gradient descent with Symmetric Minibatch Sampling of the gradient, one has to perform the analysis at the level of 2 epochs. One optimisation `step' is thus composed of $2R$ iterations of the map $x\to x+hg_{\boldsymbol{\omega}}(x)$, and so one has $\Psi_{Rh}(x)=\prod_{i=1}^{2R}\exp(h\mathcal{L}_{\widetilde{g}_{h,\boldsymbol{\omega}_i}})x$, where $\boldsymbol{\omega}_i=\boldsymbol{\omega}_{2R+1-i}$.
    From the equation for RR \cref{eq:SGDRRModEq} one has the (random) modified gradient flow of SGD-SMS (over 2 epochs) is (where $\Omega$ is defined in \cref{eq:OmegaMatrix})
    \begin{equation}\label{eq:EulerSMS}
    \widetilde{G}_{2Rh,\Omega}=2RG+\frac{h}{2}\sum_{i=1}^{2R}g'_{\boldsymbol{\omega}_i}\left(\sum_{j=1}^{i-1}g_{\boldsymbol{\omega}_j}-\sum_{j=i}^{2R}g_{\boldsymbol{\omega}_j}\right)+\mathcal{O}(h^2)=2RG-h\sum_{i=1}^{R}g'_{\boldsymbol{\omega}_i}g_{\boldsymbol{\omega}_i}+\mathcal{O}(h^2),
    \end{equation}
    after relabelling.
    Again, $X_*$ is not a fixed point of $\widetilde{G}_{2Rh,\Omega}$, nor can one find an $\alpha_{h,\boldsymbol{\omega}}$ which fulfills \cref{ass:InnerProd}. We may then find that SGD-SMS has local error
    \begin{equation}
    \begin{aligned}
\norm{\beta_{h,\boldsymbol{\omega}}(X_*)}&=h^2\norm{\sum_{i=1}^Rg'_{\boldsymbol{\omega}_i}(X_*)g_{\boldsymbol{\omega}_i}(X_*)+\mathcal{O}(h)}=\mathcal{O}(h^2),
    \end{aligned}
    \end{equation}
    and so the asymptotic error one expects to see is also $\mathcal{O}(h)$.
\end{example}

\paragraph{\textbf{Splitting Optimisers}} The BCH formalism is especially useful for optimisers which correspond to \emph{splitting integrators} \cite{SplittingMethodsODEsActa,BlanesBook}. We consider the Heavy Ball method, with steps of the form $\phi_{h}^{[A]}\circ\phi_{h}^{[B_{\boldsymbol{\omega}_i}]}$ for RM, \cref{eq:ABComp} for RR and \cref{eq:SymABComp} for SMS.
In the following we introduce $z=(x,v)$ with $Z_*=(X_*,0)$ the fixed point of the flow \cref{eq:SecondOrderFullGrad}.

\begin{example}[HB-RM]
    For the Heavy Ball Method with Robbins-Monro sampling of the gradient, one has that one optimisation step follows $\Psi_h(z)=e^{hB_{\boldsymbol{\omega}}}e^{hA}z$, and thus via the BCH formula $D_{\boldsymbol{\omega}}(h)=h(A+B_{\boldsymbol{\omega}})-\frac{h^2}{2}[A,B_{\boldsymbol{\omega}}]+\mathcal{O}(h^3)$. Hence
    the local error admits the decomposition
    $\left(e^{D_{\boldsymbol{\omega}}(h)}-e^{h\mathcal{L}_G}\right)X_*=\alpha_{h,\boldsymbol{\omega}}(X_*)+\beta_{h,\boldsymbol{\omega}}(X_*)$, where $\alpha_{h,\boldsymbol{\omega}}(X_*)=h\left(B_{\boldsymbol{\omega}}-B\right)Z_*$ is $\mathcal{O}(h)$ and $\norm{\beta_{h,\boldsymbol{\omega}}(X_*)}=\mathcal{O}(h^2)$. One may verify that $\alpha_{h,\boldsymbol{\omega}}(X_*)$ fulfills \cref{ass:InnerProd} (since $\mathbb{E}[B_{\boldsymbol{\omega}}|e_K]=B$), and thus from the definition of $B$ in \cref{eq:Operators}
    \begin{equation}
    \begin{aligned}
\norm{\alpha_{h,\boldsymbol{\omega}}(Z_*)}&=h\norm{\left(B_{\boldsymbol{\omega}}-B\right)Z_*}=h\norm{(g_{\boldsymbol{\omega}}(X_*)-G(X_*)}=\mathcal{O}(h),
    \end{aligned}
    \end{equation}
    and so the asymptotic error is $\mathcal{O}(h^{1/2})$ as for SGD-RM.
\end{example}

\begin{example}[HB-RR]
    For Heavy Ball with Random Reshuffle sampling of the gradient, one optimisation `step' is thus composed of $\Psi_{Rh}(z)=\prod_{i=1}^Re^{hB_{\boldsymbol{\omega}_i}}e^{hA}z$.
    Applying the BCH formula, one then has
$\prod_{i=1}^Re^{h(A_{\boldsymbol{\omega}_i}+B_{\boldsymbol{\omega}_i})-(h^2/2)[A,B_{\boldsymbol{\omega}_i}]+\ldots}$. Repeated applications of the BCH formula then give
    \begin{equation}\label{eq:BCHRR}
    \begin{aligned}
    \Psi_{Rh}(z)=\exp\bigg(&Rh(A+B)+h^2\left[A,\sum_{i=1}^R a_iB_{\boldsymbol{\omega}_i}\right]+\frac{h^2}{2}\sum_{i=1}^R\left[B_{\boldsymbol{\omega}_i},\sum_{j=i+1}^RB_{\boldsymbol{\omega}_j}\right]+\mathcal{O}(h^3)\bigg)z,
    \end{aligned}
    \end{equation}
    for some coefficients $(a_{i})^{R}_{i=1}$.
    Then, as the commutators obey for $i = 1,...,R$
    $$[A,B_{\boldsymbol{\omega}_i}]\rho=vg'_{\boldsymbol{\omega}_i}(x)\nabla_v\rho+(\gamma v-g_{\boldsymbol{\omega}_i}(x))\nabla_x\rho,\quad [B_{\boldsymbol{\omega}_i},B_{\boldsymbol{\omega}_j}]\rho=\gamma(g_{\boldsymbol{\omega}_j}(x)-g_{\boldsymbol{\omega}_i}(x))\nabla_v\rho,$$
    one has that,
    since $Z_*=(X_*,0)$,
    $$
    \begin{aligned}
    \alpha_{h,\boldsymbol{\omega}}(Z_*)&=h^2\sum_{i=1}^Ra_i[A,B_{\boldsymbol{\omega}_i}]Z_*+\frac{h^2}{2}\sum_{i=1}^R[B_{\boldsymbol{\omega}_i},\sum_{j=i+1}^{R}B_{\boldsymbol{\omega}_j}]Z_*\\
    &=h^2\left(-\sum_{i=1}^Ra_i\begin{bmatrix}
        g_{\boldsymbol{\omega}_i}(X_*)\\0
    \end{bmatrix}
    +\sum_{i=1,j>i}^R\begin{bmatrix}
        0\\ \gamma(g_{\boldsymbol{\omega}_j}(X_*)-g_{\boldsymbol{\omega}_i}(X_*))
    \end{bmatrix}\right),
    \end{aligned}
    $$
    which then fulfills \cref{ass:InnerProd}, with $\norm{\beta_h(X_*)}=\mathcal{O}(h^3)$. As
    $\norm{\alpha_h(X_*)}=\mathcal{O}(h^2)$
    one finds an asymptotic error of $\mathcal{O}(h^{3/2})$.
\end{example}

\begin{example}[S-HB-SMS]
Finally, we turn to (symmetric) Heavy Ball with SMS\footnote{As discussed above, we expect this to have essentially the same long term behaviour as standard Heavy Ball, but for a formal calculation it is necessary to use the Strang composition $e^{hA/2}e^{hB_{\boldsymbol{\omega}_i}}e^{hA/2}$ not the Lie-Trotter one $e^{hB_{\boldsymbol{\omega}_i}}e^{hA}$.}. One optimisation `step' is thus composed of two epochs of iterations with $$\Psi_{2Rh}(z)=\prod_{i=1}^R\left(e^{hA/2}e^{hB_{\boldsymbol{\omega}_i}}e^{hA/2}\right)\prod_{i=R}^1\left(e^{hA/2}e^{hB_{\boldsymbol{\omega}_i}}e^{hA/2}\right)z=e^{-hA/2}\left(\prod_{i=1}^Re^{hA}e^{hB_{\boldsymbol{\omega}_i}}\prod_{i=R}^1e^{hA}e^{hB_{\boldsymbol{\omega}_i}}\right)e^{hA/2}z,$$ i.e. a symmetric composition of symmetric compositions.

This map is then \emph{time-symmetric} since $\Psi_{-h}=\Psi_h^{-1}$, which is not the case for SGD-SMS (nor standard Heavy Ball SMS)\cite{BouRabee2018}. Thus its BCH expansion contains only odd powers of $h$ \cite[Section 2]{SplittingMethodsODEsActa}. Hence for S-HB-SMS one has
    $$
    \Psi_{2Rh}(x)=\exp\left(2Rh(A+B)+h^3V_3+\mathcal{O}(h^5)\right)x,
    $$
    with
    $$
V_3=\sum_{i,j,k=1}^Rd_{ijk}\left[B_{\boldsymbol{\omega}_i},\left[ B_{\boldsymbol{\omega}_j},B_{\boldsymbol{\omega}_k}\right]\right]+\sum_{i=1}^Rb_i\left[A,\left[A,B_{\boldsymbol{\omega}_i}\right]\right]+\sum_{i,j=1}^Rc_{ij}\left[B_{\boldsymbol{\omega}_i},\left[A,B_{\boldsymbol{\omega}_j}\right]\right]
    \bigg),
    $$
    for some coefficients for some coefficients $(b_{i})^{R}_{i=1},(c_{ij})^{R}_{i,j=1},(d_{ijk})^{R}_{i,j,k=1}$. Since the commutators obey
    \begin{equation*}
    \begin{aligned}
    [A,[A,B_{\boldsymbol{\omega}_i}]]\rho=-2vg'_{\boldsymbol{\omega}_i}(x)\nabla_x\rho+v^2g_{\boldsymbol{\omega}_i}''(x)\nabla_v\rho,\quad [B_{\boldsymbol{\omega}_i},[B_{\boldsymbol{\omega}_j},B_{\boldsymbol{\omega}_k}]]\rho=\gamma^2(g_{\boldsymbol{\omega}_k}(x)-g_{\boldsymbol{\omega}_j}(x))\nabla_v\rho,\\ 
    [B_{\boldsymbol{\omega}_i},[A,B_{\boldsymbol{\omega}_j}]]\rho=\gamma(g_{\boldsymbol{\omega}_i}(x)-\gamma v)\nabla_x\rho+\left(g_{\boldsymbol{\omega}_i}'(x)g_{\boldsymbol{\omega}_j}(x)+g_{\boldsymbol{\omega}_j}'(x)g_{\boldsymbol{\omega}_i}(x)-\gamma vg_{\boldsymbol{\omega}_i}'(x)\right)\nabla_v\rho,
    \end{aligned}
    \end{equation*}
    one has that,
    \begin{equation*}
    \begin{aligned}
    [A,[A,B_{\boldsymbol{\omega}_i}]]Z_*=0,\quad [B_{\boldsymbol{\omega}_i},[B_{\boldsymbol{\omega}_j},B_{\boldsymbol{\omega}_k}]]Z_*=(0,\gamma^2(g_{\boldsymbol{\omega}_k}(X_*)-g_{\boldsymbol{\omega}_j}(X_*)),\\ 
    [B_{\boldsymbol{\omega}_i},[A,B_{\boldsymbol{\omega}_j}]]Z_*=\left(\gamma g_{\boldsymbol{\omega}_i}(X_*),g_{\boldsymbol{\omega}_i}'(X_*)g_{\boldsymbol{\omega}_j}(X_*)+g_{\boldsymbol{\omega}_j}'(X_*)g_{\boldsymbol{\omega}_i}(X_*)\right),
    \end{aligned}
    \end{equation*}
    since $Z_*=(X_*,0)$. The first two terms have expectation 0 conditional on previous epochs. However, except in the case that the Hessian $g'_{\boldsymbol{\omega}}=\nabla^2f_{\boldsymbol{\omega}}$ is independent of $\boldsymbol{\omega}$, the last term does \emph{not} have expectation 0 conditioned on the previous epochs. Hence, one sets $\alpha_{h,\boldsymbol{\omega}}=0$ and $\beta_{h,\boldsymbol{\omega}}(X_*)=\mathcal{O}(h^3)$,
    giving an asymptotic error of $\mathcal{O}(h^2)$. In the case of constant Hessian, one may set $\alpha_{h,\boldsymbol{\omega}}=\mathcal{O}(h^3)$ and $\beta_{h,\boldsymbol{\omega}}(X_*)=\mathcal{O}(h^4)$, giving an error of order $\mathcal{O}(h^{5/2})$. This is confirmed by the analytic calculation in \cref{sec:ModelProblem}.
\end{example}
\begin{remark}
    In fact, almost any reasonable discretisation scheme of \cref{eq:Damped} achieves the same rates of convergence for the different batching schemes. Applying the Euler scheme, for example, gives gradient descent steps using the enlarged gradient $(x,v)\to (x+hv,v-\gamma hv+hg_{\boldsymbol{\omega}}(x))$, which implies that the expression for RR in \cref{eq:SGDRRModEq} becomes (using that $V_*=0$)
    \begin{equation*}
    \begin{aligned}
            \widetilde{G}_{h,\Omega}(Z_*)&=R\begin{bmatrix}
        V_* \\ -\gamma V_*+G(X_*)
    \end{bmatrix}+\frac{h}{2}\sum_{i=1}^R\Bigg(\sum_{j=1}^{i-1}\begin{bmatrix}
        -\gamma V_* + g_{\boldsymbol{\omega}_j}(X_*)\\ g_{\boldsymbol{\omega}_i}'(X_*)V_*+\gamma^2V_*-\gamma g_{\boldsymbol{\omega}_j}(X_*) 
    \end{bmatrix}\\
    &\hspace{6cm}-\sum_{j=i}^R\begin{bmatrix}
        -\gamma V_* + g_{\boldsymbol{\omega}_j}(X_*)\\ g_{\boldsymbol{\omega}_i}'(X_*)V_*+\gamma^2V_*-\gamma g_{\boldsymbol{\omega}_j}(X_*) 
    \end{bmatrix}\Bigg)+\mathcal{O}(h^2)\\
    &=\frac{h}{2}\sum_{i=1}^R\left(\sum_{j=1}^{i-1}\begin{bmatrix}
        g_{\boldsymbol{\omega}_j}(X_*) \\ -\gamma g_{\boldsymbol{\omega}_j}(X_*)
    \end{bmatrix}-\sum_{j=i}^R\begin{bmatrix}
        g_{\boldsymbol{\omega}_j}(X_*) \\ -\gamma g_{\boldsymbol{\omega}_j}(X_*)
    \end{bmatrix}\right)+\mathcal{O}(h^2),
        \end{aligned}
    \end{equation*}
    which has expectation 0 conditional on previous epochs. Consequently, one has $\alpha_{h,\boldsymbol{\omega}}(X_*)=\mathcal{O}(h^2)$ and $\beta_{h,\boldsymbol{\omega}}(X_*)=\mathcal{O}(h^3)$ as for Heavy Ball giving the asymptotic error of $\mathcal{O}(h^{3/2})$.

    Similarly, the expression for SMS in \cref{eq:EulerSMS} becomes
    $$
    \widetilde{G}_{h,\Omega}(Z_*)=h\sum_{i=1}^R\begin{bmatrix}
        g_{\boldsymbol{\omega}_i}(X_*) \\ -\gamma g_{\boldsymbol{\omega}_i}(X_*)
    \end{bmatrix}+\mathcal{O}(h^2)=Rh\begin{bmatrix}
        G(X_*) \\ -\gamma G(X_*)
    \end{bmatrix}+\mathcal{O}(h^2)=0+\mathcal{O}(h^2),
    $$
    and so one has $\alpha_{h,\boldsymbol{\omega}}(X_*)=0$ and $\beta_{h,\boldsymbol{\omega}}(X_*)=\mathcal{O}(h^3)$ as for Heavy Ball, giving an asymptotic RMSE bias of $\mathcal{O}(h^2)$.
\end{remark}

\subsection{Overview of Analysis}
The rates suggested by the analysis presented here are confirmed by experiments in \cref{sec:Numerics}, analytical calculations in \cref{sec:ModelProblem} and the convergence results of \cref{sec:convergence}, which conclusively demonstrate the superiority of (symmetric) Heavy Ball as a stochastic gradient scheme.

It is thus clear that on the one hand standard gradient descent cannot be written as a splitting method based on exact submaps. Without using a momentum-based scheme, only a higher order method can exploit symmetry. For example, one could consider a Runge-Kutta methods for \eqref{eq:GD_cont}, but they require multiple gradients per step, so they would only be desirable in high-accuracy regimes. Further, they have undesirable properties in the full-gradient setting, in particular, in \cite{SanzZyg2020} they show there exists Runge-Kutta schemes that fail to be contractive for any choice of the timestep.

\begin{remark}\label{rem:varRed}
In \cref{alg:GenAlg} and elsewhere we have assumed that $N/n$ is an integer ($n$ is the size of each minibatch, $N$ the number of data points). However, quite often, this is not the case in practice. 

In the case where $n$ does not exactly divide $N$, the final minibatch contains $n_{R}=N - n\lceil N/n\rceil < n$ datapoints. It is standard to simply average the final stochastic gradient approximation with this smaller $n_R$, i.e. using $n_R^{-1}\sum_{j=1}^{n_R}f_{\boldsymbol{\omega}_{Rj}}$. This means one can no longer understand the optimiser as a splitting method, and indeed in experiments using this procedure destroys the higher-order convergence that is shown in this work. Rather, if the final minibatch is of size $n_{R}$ for $n_{R}<n$ one must premultiply the stochastic gradient approximation by $n_{R}/n$, so that the sum of the stochastic gradients used in the epoch remains the same as $\nabla F$ (up to a multiplicative constant). This may alternatively be seen as a variance reduction technique.
\end{remark}

\section{Example with analytic computation}\label{sec:ModelProblem}
To investigate the minibatching strategies/splitting methods introduced in \cref{sec:randomisation_strategies} we first consider the application of RM, RR, SMS to stochastic gradient methods with and without momentum to a simple 1D problem from \cite{Leimkuhler2016}. Given a dataset $Y=\{y_i\}_{i=1}^N$ with $y_i\in\R$ for $i=1,...,N$ we define the objective function
\begin{equation}\label{eq:ModelProblem}
F(X)=\frac{1}{2}\sum_{i=1}^N\sigma^{-2}(X-y_i)^2. 
\end{equation}

Note that this optimisation problem can be formed via considering $\{y_i\}_{i=1}^N$ independent and identically distributed random variables under a parametrised model $y_{\vert X}\sim\mathcal{N}(X,\sigma_y^2)$, and formulating an inference problem for the parameter $X$, applying either a uniform prior as in \cite{Leimkuhler2016} or Gaussian prior (and rescaling) as in \cite{Vollmer2016}. Note also that $d$-dimensional versions of such problems which result in $X\sim\mathcal{N}(\boldsymbol{m},C)$ with $C\in\R^{d\times d}$ positive definite admit diagonalisation, giving uncoupled copies of the 1D case \cref{eq:ModelProblem}, reinforcing the relevance of this problem, since such diagonalisation commutes with the optimisation algorithms considered here.
The gradient indeed takes the form of a finite sum as in \cref{eq:FiniteSum}, for $x \in \mathbb{R}^{d}$ defined by
\begin{equation}\label{eq:ModelProblem}
    \nabla F(x)=\frac{1}{N}\sum_{i=1}^N N\sigma_i^{-2}(x-y_i),
\end{equation}
where we take $\sigma_i=\sigma$ constant. The continuous gradient flow from \cref{eq:GD} then takes the form
$$
    \frac{dX_t}{dt}=-N\sigma^{-2}(X_t-X_*),\quad X_*\equiv\frac{1}{N}\sum_{j=1}^{N}y_{i}.
$$
The stochastic gradient in \cref{eq:SGD} is (cf. \cref{eq:stochgrad}) then generated, for $x\in \mathbb{R}^{d}$ and $i = 1,...,m$ via
$$
\nabla f_{\boldsymbol{\omega_{i}}}(x)(x)=N\sigma^{-2}(x-\widehat{y}_{\boldsymbol{\omega}_i}),\quad \widehat{y}_{\boldsymbol{\omega}_i}\equiv\frac{1}{n}\sum_{j=1}^{n}y_{\omega_{ij}},
$$
 with $n\ll N$, and the vectors $\boldsymbol{\omega}_{i}$ are generated via either SMS or RR or RM randomisation procedure according to \cref{alg:GenAlg}. For ease in the following we will generically denote the stochastic estimate at iteration $k\in \mathbb{N}$ as $\widehat{y}_{k}$, without risk of confusion since the index is an integer not a vector. 

\paragraph{Results for first-order dynamics}
Applying SGD to the flow for the model problem gives iterates (after preconditioning $h\gets h\sigma^{2}/N$)
$$
x_{K+1}=(1-h)x_{K}+h\widehat{y}_K=(1-h)^{K+1}x_{0}+h\sum_{k=0}^K(1-h)^{K-k}\widehat{y}_k.
$$
For simplicity, we set $x_0=0$. In the limit ${K\to\infty}$, all SGD (-RM,-RR,-SMS) schemes converge to  $\bar{y}=X_*$ and so, following \cite{shaw2025random}, one can examine the MSE $\lim_{K\to\infty}\mathbb{E}[\|x_K-X_*\|^2]$ which may be identified with the asymptotic variance $\mathbb{V}[x_{\infty}]$.

\subsection{SGD for the model problem}
For SGD-RM, one has simply that $\mathbb{V}[x_{\infty}]=h^2V/(1-(1-h)^2)=Vh/2+\mathcal{O}(h^2)$, where $V=\mathbb{V}[\hat{y}]$. For SGD-RR, after grouping into epochs, one has 
$$
\mathbb{V}[x_{\infty}]=\sum_{t=0}^{\infty}(1-h)^{2Rt}\mathbb{V}[\widetilde{u}],\quad \mathbb{V}[\widetilde{u}]=\frac{V}{R-1}\left[\frac{Rh(1-(1-h)^{2R})}{2-h}-(1-(1-h)^{R})^2\right]
$$
so that the asymptotic MSE is $\mathbb{V}[\widetilde{u}]/(1-(1-h)^{2R})=Vh^3\frac{R \left(R + 1\right)}{24}+\mathcal{O}(h^4)$.

A calculation along similar lines for SMS shows that the iterates obey $\mathbb{E}[\|x_K-X_*\|^2]=\sum_{t=0}^{\infty}(1-h)^{4tR}\mathbb{V}[\widetilde{u}]$ where
\begin{equation*}
    \mathbb{V}[\widetilde{u}]=h^2\sum_{j,j'=0}^{2R-1}(1-h)^{(j+j')}
\text{cov}(\widehat{y}_{j},\widehat{y}_{j'})
=h^2\sum_{j,j'=0}^{R-1}+2h^2\sum_{j=0}^{R-1}\sum_{j=R}^{2R-1}+h^2\sum_{j,j'=R}^{2R-1},
\end{equation*}
(we suppress the repeated summands for brevity)
which, using that $\text{cov}(\widehat{y}_{j},\widehat{y}_{j'})=-V/(R-1)$ for $j'\neq j$ (where $V=\mathbb{E}[\|\widehat{y}-\bar{y}\|^2]$) and that $\widehat{y}_{j}=\widehat{y}_{2R-1-j}$ gives that
$$
\mathbb{V}[\widetilde{u}]=\frac{V}{R-1}\left(\frac{Rh(1-(1-h)^{4R})}{2-h}-(1-(1-h)^{2R})^2+2h^2R^2(1-h)^{2R-1}\right).
$$
This then gives the MSE as $\mathbb{V}[\widetilde{u}]/(1-(1-h)^{4R})=\frac{h^5VR \left(R + 1\right) \left(2 R - 1\right) \left(2 R + 1\right)}{180}+\mathcal{O}(h^6)
$.

\subsection{Model problem with momentum}
 The dynamics in \cref{eq:Damped} with a stochastic gradient for the model problem, after rescaling $\gamma,V$ by $\sigma/N$, and $t$ by $N/\sigma$ may be brought to the form \begin{equation*}
    \begin{split}
    \frac{dX}{dt} &= V\\
    \frac{dV}{dt} &= -(X-\widehat{y}) - \gamma V,
\end{split}
\end{equation*}
which has exact solution 
$$
Z(h)=\begin{bmatrix}
    X(h)\\
    V(h)
\end{bmatrix}=e^{hA}Z(0)+(I-e^{hA})\begin{bmatrix}
    \widehat{y}\\
    0
\end{bmatrix},\quad A=\begin{bmatrix}
    0&1\\
    -1&-\gamma
\end{bmatrix}.
$$
 In order to draw conclusions about general momentum-based optimisers, we consider the stochastic gradient bias incurred when using the exact solution, since all relevant optimisations schemes, in the $h\to0$ limit, will converge to the exact map.
It may be seen that after $K+1$ iterations, starting from $z_0=0$, with different stochastic gradients $\widehat{u}_k\equiv (\widehat{y}_k,0)^T$ one has
$$
z_{K+1}=e^{(K+1)hA}z_0+(I-e^{hA})\sum_{k=0}^Ke^{(K-k)hA}\widehat{u}_k=(I-e^{hA})\sum_{k=0}^Ke^{(K-k)hA}\widehat{u}_k.
$$
$A$ has eigenvalues $\lambda_{\pm}=-\gamma/2\pm\sqrt{(\gamma/2)^2-1}$; assuming $\gamma\neq2$, one may diagonalise $A$ and show that $\mathbb{E}[z_{\infty}]\equiv\lim_{K\to\infty}\mathbb{E}[z_{K}]=(X_*,0)^T$. We then consider the limiting MSE $\lim_{K\to\infty}\mathbb{E}[\|x_{K}-X_*\|^2]$, i.e. the upper-left entry of the asymptotic covariance matrix of $z_K$, $\mathbb{V}[Z_\infty]$.

\paragraph{\textbf{MSGD-RM}}
For RM, the asymptotic variance is simply
$$
\mathbb{V}[Z_\infty]=(I-e^{hA})\sum_{k=0}^\infty e^{khA}\mathbb{V}[\widehat{u}]e^{khA^T}(I-e^{hA})^T,\quad \mathbb{V}[\widehat{u}]=\begin{bmatrix}
    V & 0\\ 0 & 0
\end{bmatrix},
$$
where $V=\mathbb{V}[\widehat{x}]$, since the stochastic gradients are independent between iterations. Tedious calculations give that the upper-left entry is then
$$
\frac{V}{(\lambda_+-\lambda_-)^2}\left[\frac{\lambda_-^2(1-e^{h\lambda_+})^2}{1-e^{2h\lambda_+}}-\frac{2(1-e^{h\lambda_+})(1-e^{h\lambda_-})}{1-e^{-\gamma h}}+\frac{\lambda_+^2(1-e^{h\lambda_-})^2}{1-e^{2h\lambda_-}}\right]=\frac{Vh}{2\gamma}+\mathcal{O}(h^2).
$$

\paragraph{\textbf{MSGD-RR}}
For RR, the asymptotic variance may be broken up into epochs, using independence of the stochastic gradients between epochs
$$
\mathbb{V}[Z_\infty]=\sum_{t=0}^\infty e^{RthA}\mathbb{V}[\widetilde{u}]e^{RthA^T},
$$
where $\mathbb{V}[\widetilde{u}]$ is given by
$$
\mathbb{V}[\widetilde{u}]=\frac{1}{R-1}(I-e^{hA})\left[\sum_{k=0}^{R-1} Re^{khA}\mathbb{V}[\widehat{u}]e^{khA^T}-\sum_{k,k'=0}^{R-1} e^{khA}\mathbb{V}[\widehat{u}]e^{k'hA^T}\right](I-e^{hA})^T,\quad \mathbb{V}[\widehat{u}]=\begin{bmatrix}
    V & 0\\ 0 & 0
\end{bmatrix},
$$
using that $\mathbb{V}[\widetilde{u}_j|\widetilde{u}_i]=-V/(R-1)$. Diagonalising, one may obtain the in-epoch variance $\mathbb{V}[\widetilde{u}]$ with some difficulty, and it is then easy enough to get that the upper-left entry of $\mathbb{V}[Z_\infty]$
\begin{equation*}
\begin{aligned}
\frac{V}{(R-1)(\lambda_+-\lambda_-)^2}&\Bigg[\frac{R\lambda_-^2(1-e^{h\lambda_+})^2}{1-e^{2h\lambda_+}}-\frac{\lambda_-^2(1-e^{Rh\lambda_+})^2}{1-e^{2Rh\lambda_+}}
-\frac{2R(1-e^{h\lambda_+})(1-e^{h\lambda_-})}{1-e^{-\gamma h}}+\\
&\frac{2(1-e^{Rh\lambda_+})(1-e^{Rh\lambda_-})}{1-e^{-\gamma Rh}}+\frac{R\lambda_+^2(1-e^{h\lambda_-})^2}{1-e^{2h\lambda_-}}-\frac{\lambda_+^2(1-e^{Rh\lambda_-})^2}{1-e^{2Rh\lambda_-}}\Bigg],
\end{aligned}
\end{equation*}
which a Taylor expansion shows to be $\frac{VR(R+1)h^3}{24\gamma}+\mathcal{O}(h^4)$.

\paragraph{\textbf{MSGD-SMS}}
Similarly, for SMS, the asymptotic variance may be broken up into multiples of 2 epochs, since the stochastic gradients are independent outside the symmetrically batched epochs
$$
\mathbb{V}[Z_\infty]=\sum_{t=0}^\infty e^{2RthA}\mathbb{V}[\widetilde{u}]e^{2RthA^T},
$$
where $\mathbb{V}[\widetilde{u}]$ is given by
\begin{equation*}
\begin{aligned}
\mathbb{V}[\widetilde{u}]=\frac{1}{R-1}(I-e^{hA})&\Bigg[\sum_{k=0}^{R-1} Re^{khA}\mathbb{V}[\widehat{u}]e^{khA^T}-\sum_{k,k'=0}^{R-1} e^{khA}\mathbb{V}[\widehat{u}]e^{k'hA^T}\\
&+\sum_{k=0}^{R-1} Re^{khA}\mathbb{V}[\widehat{u}]e^{-khA^T}e^{(2R-1)hA^T}-\sum_{k=0}^{R-1}\sum_{k'=R}^{2R-1} e^{khA}\mathbb{V}[\widehat{u}]e^{k'hA^T}\\
&+\sum_{k=0}^{R-1} Re^{(2R-1)hA}e^{-khA}\mathbb{V}[\widehat{u}]e^{khA^T}-\sum_{k=R}^{2R-1}\sum_{k=0}^{R-1} e^{khA}\mathbb{V}[\widehat{u}]e^{k'hA^T}\\
&+e^{RhA}\left(\sum_{k=0}^{R-1} Re^{khA}\mathbb{V}[\widehat{u}]e^{khA^T}-\sum_{k=0}^{R-1}\sum_{k=0}^{R-1} e^{khA}\mathbb{V}[\widehat{u}]e^{k'hA^T}\right)e^{RhA}
\Bigg](I-e^{hA})^T,
\end{aligned}
\end{equation*}
using that $\mathbb{V}[\widetilde{u}_j|\widetilde{u}_i]=-V/(R-1)$, and that $\widehat{u}_j=\widehat{u}_{2R-1-j}$. Diagonalising, one may obtain the in-epoch variance $\mathbb{V}[\widetilde{u}]$ with some difficulty. It is then possible to obtain a cumbersom expression for the upper-left entry of $\mathbb{V}[Z_\infty]$
which a Taylor expansion shows to be 
$$\frac{R V h^{5} \left(R + 1\right) \left(2 R - 1\right) \left(2 R + 1\right) \left(\gamma^{2} + 1\right)}{180 \gamma}
+\mathcal{O}(h^6).$$

\begin{remark} Hence, for RM, RR, SMS, whether one uses SGD or a momentum-based method, one has convergence in the MSE of order $\mathcal{O}(hR),\mathcal{O}((hR)^3),\mathcal{O}((hR)^5)$ respectively. One thus recovers the bias of \cref{thm:SGDRM} for SGD-RM, but these rates are overly optimistic for the other SGD methods, as shown in \cite{rajput2022permutation,Cha2023}.
In the case of momentum-RR, this rate matches tightly our result  in \cref{theorem:SGPolyak}. In practice, the constant variance $\sigma^2$ between batches is unrealistic, and a simple experiment with variable $\sigma_i^2$ in \cref{eq:ModelProblem} shows that the $\mathcal{O}(h^2)$ bound in \cref{theorem:SGPolyak} is also likely tight for momentum-SMS (see \cref{fig:modelproblem}).

\begin{figure}
    \centering
    \includegraphics[width=0.5\linewidth]{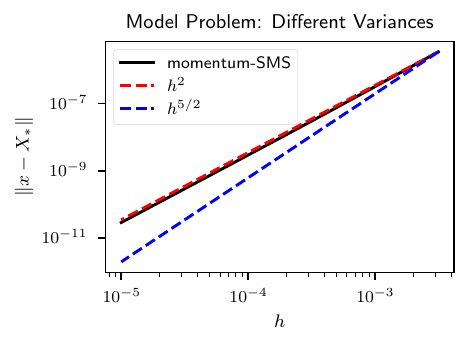}
    \caption{An experiment for the model problem \cref{eq:ModelProblem} with $\sigma_i$ for $i = 1,...,N$ not constant shows that the bias in the RMSE is no longer $\mathcal{O}(h^{5/2})$ but rather $\mathcal{O}(h^{2})$. We take $N=5$ and set $\boldsymbol{\sigma}^2=[\sigma_1^2,\ldots,\sigma_N^2]$, taking $\boldsymbol{\sigma}^2=[2.5,1.5,0.05,0.15,0.1]$ with $x_i=i,i=1,\ldots 5$, and use the Euler method to solve the resulting system \cref{eq:Damped}.}
    \label{fig:modelproblem}
\end{figure}
\end{remark}

\section{Convergence Guarantees}\label{sec:convergence}
As in \cite{Mattingly2002} we consider the Lyapunov function $\mathcal{V}:\mathbb{R}^{d} \times \mathbb{R}^{d} \to \mathbb{R}$, defined for $(x,v)\in \mathbb{R}^{2d}$ by
\begin{equation}\label{eq:Lyapunov}
    \mathcal{V}(x,v) = F(x) - F(X_{*}) + \frac{\gamma^{2}_{h}}{4}\|x-X_{*}\|^{2} + \frac{\gamma_{h}}{2}\left\langle x-X_{*},v\right\rangle + \frac{1}{2}\|v\|^{2},
\end{equation}
where $\gamma_{h} \to \gamma$ as $h\to 0$ and depends on the discretisation (see \cite{LePaWh24}) and $\|\cdot\|$ is the standard Euclidean 2-norm. For the Heavy Ball method we choose $\gamma_{h} = \frac{1-\eta}{h\eta}$ in the convergence analysis, where $\eta = e^{-h\gamma}$.

We have the following equivalency conditions for this Lyapunov function \cite{Mattingly2002}.
\begin{lemma}\label{lemma:equiv}
For $(x,v) \in \mathbb{R}^{d} \times \mathbb{R}^{d}$ we have that
    \begin{align*}
        \mathcal{V}(x,v) \geq \frac{1}{8}\|v\|^{2} + \frac{\gamma^{2}_{h}}{12}\|x-X_{*}\|^{2}
    \end{align*}
    and
    \begin{align*}
        \mathcal{V}(x,v) \leq F(x) - F(X_{*}) + \frac{\gamma^{2}_{h}}{2}\|x-X_{*}\|^{2} +\|v\|^{2}.
    \end{align*}
\end{lemma}

This choice of Lyapunov function or variants of it are popular for proving the convergence of momentum based optimisation schemes (see \cite{Fazlyab2018,SanzZygOpt,KaroniOpt,Bach2023}
and references therein).

In addition to the Lyapunov function \cref{eq:Lyapunov}, we shall also require some assumptions on $F$ and the component functions $f_{\boldsymbol{\omega}}$.

\begin{assumption}[Components are convex and $L$-smooth]\label{assum:smoothness_sg}
For some positive constants $\mu,L \in \mathbb{R}_{+}$ we assume the potential $F:\mathbb{R}^{d} \to \mathbb{R}$ is of the form \cref{eq:FiniteSum} and is continuously differentiable, $L$-smooth and $\mu$-strongly convex for some positive constants $\mu,L \in \mathbb{R}_{+}$. In addition we assume that for any instance of $\boldsymbol{\omega} \subset \{1,2,...,N \}$ such that $|\boldsymbol{\omega}| = n<N$, $f_{\boldsymbol{\omega}}:\mathbb{R}^{d}\to\mathbb{R}$ defined by $f_{\boldsymbol{\omega}} = \frac{1}{n}\sum_{i \in \boldsymbol{\omega}}f_{i}$ are continuously differentiable, convex and $L$-smooth.
\end{assumption}  

\begin{theorem}\label{theorem:SGPolyak}
Under \cref{assum:smoothness_sg} consider the stochastic gradient Polyak Heavy Ball scheme \cref{eq:Polyak} with $\gamma > 0$, $0<h< \min\{1/2R\gamma,1/2R\sqrt{L}\}$. Assuming the stochastic gradients also satisfy \cref{assum:stochastic_gradient}, then we have that the iterates $(x_{i},v_{i})_{i\in \mathbb{N}}$ for $k \in \mathbb{N}$ satisfy;
    \begin{align*}
    \mathbb{E}[F(x_{2kR}) - F(X_{*})] \leq \mathbb{E}\mathcal{V}(x_{2kR},v_{2kR}) \leq (1-chR)^{k}\mathcal{V}(x_{0},v_{0}) + \frac{C(\mu,L,\gamma)h^{3}R^{2}\sigma_*^{2}}{c},
\end{align*}
when the stochastic gradients are generated according using Random Reshuffling; and
    \begin{align*}
    \mathbb{E}[F(x_{2kR}) - F(X_{*})] \leq \mathbb{E}\mathcal{V}(x_{2kR},v_{2kR}) \leq (1-chR)^{k}\mathcal{V}(x_{0},v_{0}) + \frac{C(\mu,L,\gamma)h^{4}R^{3}\sigma_*^{2}}{c},
\end{align*}
when the stochastic gradients are generated according to the Symmetric Minibatching Strategy. Here,
\begin{align*}
    c&= \frac{\min\left\{\gamma,\mu\gamma/\gamma^{2}_{h}\right\}}{4} - C(\mu,L,\gamma)hR,
\end{align*}
where $C(\mu,L,\gamma)$ is a constant depending on $m,L$ and $\gamma$.
\begin{proof}
    Follows from \cref{lemma:SGD,lem:randomization}
\end{proof}
\end{theorem}

\begin{corollary}\label{corr:SGPolyak}
Under the assumptions of \cref{theorem:SGPolyak} we have that for $\epsilon > 0$ setting $h = \mathcal{O}(\epsilon^{1/3}/R)$ that $\mathbb{E}[F(x_{2kR})-F(X_{*})] < \epsilon$ in $$\mathcal{O}\left(\frac{1}{\epsilon^{1/3}}\log{\left(\frac{\mathcal{V}(x_{0},v_{0})}{\epsilon}\right)^{+}}\right),$$
epochs with the random reshuffling stochastic gradient policy. 

Setting $h = \mathcal{O}(\epsilon^{1/4}/R)$ we have that $\mathbb{E}[F(x_{2kR})-F(X_{*})] < \epsilon$ in $$\mathcal{O}\left(\frac{1}{\epsilon^{1/4}}\log{\left(\frac{\mathcal{V}(x_{0},v_{0})}{\epsilon}\right)^{+}}\right)$$  
epochs for the symmetric minibatching stochastic gradient policy. 
\end{corollary}
\begin{remark}
    Although we prove this for the Polyak Heavy Ball discretisation \cref{eq:Polyak}, using the same argument one can arrive at the same bound for the Nesterov Accelerated Gradient method \cref{eq:Nesterov} using the parameterisation of \cite{SanzZygOpt}. This is reinforced by the experiments in \cref{sec:Numerics}, which show little difference between the Nesterov and Heavy Ball schemes.
\end{remark}

\begin{remark}
Since it can be shown that $\sigma_*^2\propto R$ (see \cite[Lemma 2.1]{shaw2025random} and \cite[Prop. 3.10]{Gower2019}), we can compare the results of \cref{theorem:SGPolyak} to the bound of \cite[Thm. 2]{Mishchenko2020} for SGD-RR of $\mathcal{O}((Rh)^2)$. We see that we are able to reduce the stochastic gradient bias to $\mathcal{O}((Rh)^{4})$ for SMS and to $\mathcal{O}((Rh)^{3})$ using RR, when using a momentum-based scheme in conjunction with these strategies. Note that, as confirmed by the analytic calculation in \cref{sec:ModelProblem}, HB-RM (or any other momentum-based scheme) has the same $\mathcal{O}(Rh)$ stochastic gradient bias as SGD-RM \cref{thm:SGDRM}.

    These improved results result in improved dependence on the desired accuracy $\epsilon$ in the complexity guarantees in \cref{corr:SGPolyak} compared to the respective complexity guarantees for SGD-RM and SGD-RR. 
\end{remark}

\section{Numerical Experiments}\label{sec:Numerics}
To verify the bias rates found by the three different routes (analytical calculation for model problem in \cref{sec:ModelProblem}, formal calculation  based on analysis of splitting methods in \cref{sec:splitting_methods}, and rigorous Lyapunov-type bounds in \cref{sec:convergence}), we perform some numerical experiments. Code to reproduce the plots is available via a repository hosted on \href{https://github.com/lshaw8317/symbatchopt}{\texttt{GitHub}}.
We consider SGD, Polyak's Heavy Ball method and Nesterov's method combined with the three different strategies (RM, RR and SMS), for some logistic regression problems adapted\footnote{We use a different prior in this work, and a second, minor, difference is that for the problem with simulated data we use 1024 datapoints, rather than $10^4$ as in \cite{Casas2022}.} from \cite{Casas2022}, with three real datasets (Chess, CTG, StatLog) and one simulated dataset (SimData). In this case, the objective takes the form for $X\in \mathbb{R}^{d}$ 
\begin{equation}\label{eq:LogRegF}
    F(X)=\frac{1}{N}\sum_{i=1}^N\frac{\lambda}{2}\|X\|_2^2 -z_iX^T\widetilde{{y}}_i+\log\left[1+\exp\left(X^T\widetilde{{y}}_i\right)\right],
\end{equation}
where the datapoints $y_i=[z_i,\widetilde{y}_i]^T$ are composed of labels $z_i\in\{0,1\}$  and feature variables $\widetilde{{y}}_i\in\R^d$ for $i = 1,...,N$. We set $\lambda=L/\sqrt{N}$, where $L=\|Y^TY\|_2/4N$ where the matrix $Y$ has entries $Y_{ij}=y_{ij},j=1,\ldots d$ (i.e $\|Y^TY\|_2$ is the maximum eigenvalue of $Y^TY$). We determine the true minimiser up to machine precision using optimally-tuned Nesterov with the full gradient, and examine the RMSE $\|x_k-X_*\|_2$ averaged over 100 stochastic gradient realisations. For all experiments, $R=8$.

For the bias plots we use an even number of epochs $n_e=2\lceil\max(5/h,500)/2\rceil$ for each timestep $h$, start from $x_0=X_*$ and plot the error in the final iterate $\|x_K-X_*\|_2$. Results are shown in \cref{fig:Exps_bias}, where in fact the convergence for the Nesterov and Heavy Ball methods is of order $\mathcal{O}(h^{5/2})$ (although for even smaller $h$, one does eventually see the theoretical $\mathcal{O}(h^{2})$) since the Hessian for the logistic regression is roughly constant between batches (due to the Bernstein-von Mises theorem, see \cite[Appendix A]{Casas2022}).

Since in practice a decreasing stepsize schedule is often used, we examine the convergence progress of the different methods according to
the schedule $h_k^{-1}=L(1 + \delta\max(0, k-20R)/R)$ where $\delta$ varies between $1/3-1/7$ for the different datasets. Again the momentum-based SMS methods converge closest to the minimiser, reinforcing their superiority (see \cref{fig:Exps_DSS}).

It should be emphasised that \emph{all methods have the same cost}, and hence the considerable improvement in performance for momentum-based SMS methods comes at no computational disadvantage.

\begin{figure}
\centering
\begin{subfigure}{.48\textwidth}
\includegraphics[width=\textwidth]{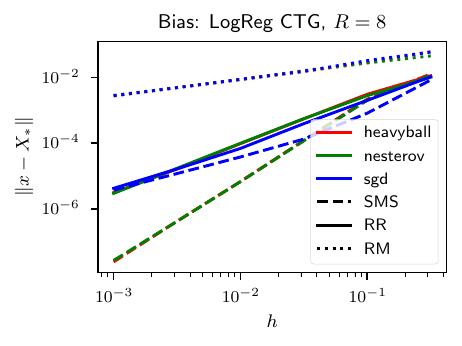}
\end{subfigure}
\begin{subfigure}{.48\textwidth}
\includegraphics[width=\textwidth]{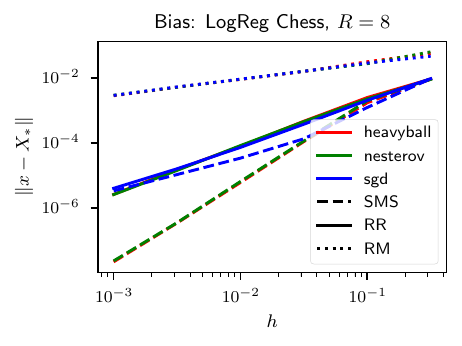}
\end{subfigure}
\begin{subfigure}{.48\textwidth}
\includegraphics[width=\textwidth]{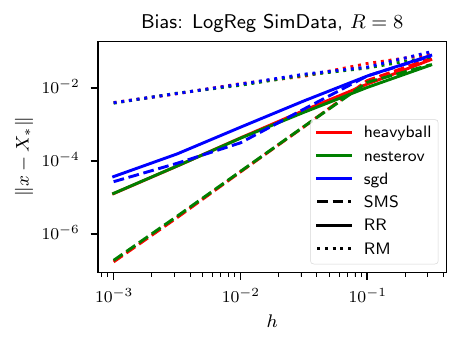}
\end{subfigure}
\begin{subfigure}{.48\textwidth}
\includegraphics[width=\textwidth]{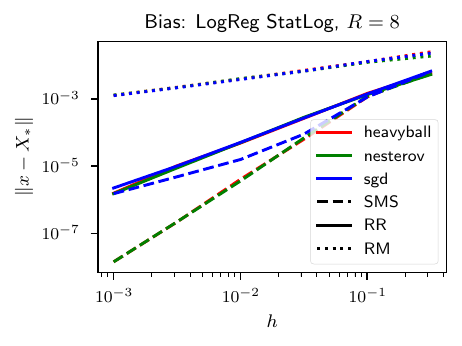}
\end{subfigure}
\caption{The error norm $\|x-X_*\|$ is the RMSE (in the Euclidean 2-norm) over 100 independent stochastic gradient realisations. Note that the final batch in each epoch is of a different size to the other batches for all the datasets except SimData, and that no reduction of order of the bias is observed (as would be the case if one had not reweighted the gradients correctly as described in \cref{rem:varRed}).}
\label{fig:Exps_bias}
\end{figure}

\begin{figure}
\centering
\begin{subfigure}{.48\textwidth}
\includegraphics[width=\textwidth]{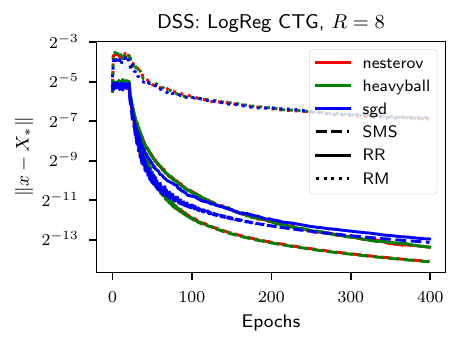}
\end{subfigure}
\begin{subfigure}{.48\textwidth}
\includegraphics[width=\textwidth]{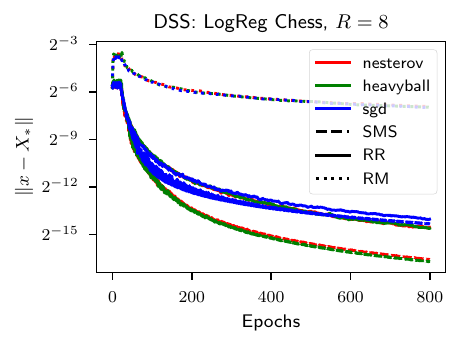}
\end{subfigure}
\begin{subfigure}{.48\textwidth}
\includegraphics[width=\textwidth]{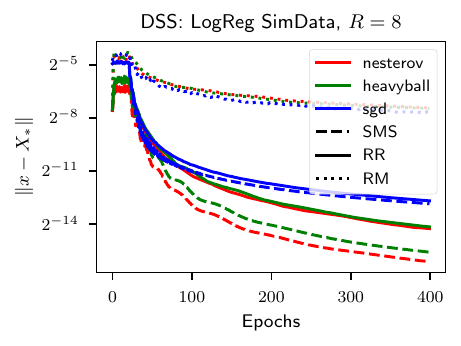}
\end{subfigure}
\begin{subfigure}{.48\textwidth}
\includegraphics[width=\textwidth]{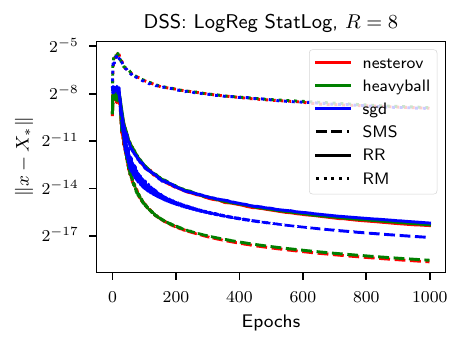}
\end{subfigure}
\caption{$\delta=1/3$ for SimData, $\delta=1/4$ for CTG, $\delta=1/6$ for StatLog and $\delta=1/7$ for Chess.}
\label{fig:Exps_DSS}
\end{figure}

\section{Conclusion and future directions}

To conclude, we establish a clear link between stochastic gradient optimiser minibatching strategies and splitting methods for ODEs. Through this perspective we investigate a minibatch strategy motivated by Strang's splitting and provide quantitative guarantees for the Heavy Ball method, which explains the considerable performance gains observed in practice.  We show that these bounds are tight on a Gaussian toy example and provide formal analysis for these randomisation strategies based on techniques from the splitting method literature.

Further work could examine whether assuming strong convexity of the component functions $f_{\boldsymbol{\omega}}$ leads to less restrictive requirements on the stepsize $h$, as is the case for SGD-RR (cf. \cite[Thms. 1 and 2]{Mishchenko2020}). In a similar vein, results in the non-convex setting under a bounded variance assumption would be of interest, since a major advantage of stochastic gradients is exploration and reaching ``good" local minima in non-convex settings \cite{Mishchenko2020}. From a practical viewpoint, it may be interesting to consider how the improved minibatching strategies behave in non-convex scenarios, for example, in neural network training. In particular, recent work \cite{beneventano2023trajectories} claims that SGD-RR traverses flat areas significantly faster than SGD-RM. One may expect that the same behaviour holds for SMS.

Often in practice, one uses Polyak-Ruppert averaging to achieve a more accurate approximation of the minimiser (see \cite{ruppert1988efficient,polyak1990new} and \cite{Dieuleveut2020} for some recent analysis). A natural extension would be to examine the asymptotic error of Polyak-Ruppert averaging for the minibatching strategies discussed in this work. A generalisation of the modified equation analysis in \cref{sec:SplittingAnalysis} following the framework of e.g. \cite{Abdulle2015} would likely be illuminating.

Finally, there is a extremely rich literature on splitting methods (and related fields, such as geometric integration). One topic of recent interest is the use of complex timesteps to attain higher order integrators, which may be of interest to optimisation researchers too, now that the connection to splitting methods is clear
\cite{Blanes2013,Castella2009,Blanes2022,Blanes2022symmetric,Bernier2023,Hansen2009,Casas2021}.

\section*{Acknowledgements} 
The authors thank Yuansi Chen and Andrea Nava for helpful discussions.

\bibliographystyle{amsplain}
\bibliography{refs}
\appendix

\section{Relevant Results for Splitting Methods}
\subsection{Extra definitions}
\begin{definition}\label{def:ExpLog}
The exponential and logarithm of a {linear, bounded} operator $A$ may be formally defined as \cite{Rossmann2006}
\begin{equation}
\exp(A)=\sum_{n=0}^{\infty}\frac{1}{n!}A^n=I+A+\frac{A^2}{2}+\ldots,\qquad\log(1-A)=-\sum_{n=1}^{\infty}\frac{A^n}{n}=-A-\frac{A^2}{2}-\frac{A^3}{3}-\ldots \label{eq:ExpLogDefs}.
\end{equation}
Note that then the derivative $\left.(d/dt)\right|_{t=0}\exp(tA)=A$.
\end{definition}
\subsection{Global Convergence Error}\label{sec:AsymptoticConvergenceError}

Consider applying $K+1$ steps of the stochastic optimiser defined by $\Psi_h$, starting from $x$. Then
\begin{gather*}
\begin{aligned}
        \norm{\Psi_h^{K+1}(x)-\phi_h^{K+1}(x)}&=\norm{\Psi_h(\Psi_h^{K}(x))-\Psi_h(\phi_h^{K}(x))+\Psi_h(\phi_h^{K}(x))-\phi_h(\phi_h^{K}(x))}\\
        &=\norm{\Psi_h(\Psi_h^{K}(x))-\Psi_h(\phi_h^{K}(x))+\alpha_{h,{\boldsymbol{\omega}}}(\phi_h^{K}(x))+\beta_{h,{\boldsymbol{\omega}}}(\phi_h^{K}(x))}\\
        & \leq \norm{\Psi_h(\Psi_h^{K}(x))-\Psi_h(\phi_h^{K}(x))+\alpha_{h,{\boldsymbol{\omega}}}(\phi_h^{K}(x))}+\norm{\beta_{h,{\boldsymbol{\omega}}}(\phi_h^{K}(x))}.
\end{aligned}
\end{gather*}
If we let $x=X_*$ then $\phi_h(X_*)=X_*$ and we write $e_K=\Psi_h^K(X_*)-X_*$ then one has
\begin{gather*}
\begin{aligned}
    \norm{e_{K+1}}&=\left(\norm{\alpha_{h,{\boldsymbol{\omega}}}(X_*)}^2+2\mathbb{E}\langle \alpha_{h,{\boldsymbol{\omega}}}(X_*),\Psi_h(\Psi_h^{K}(x))-\Psi_h(X_*)\myrangle
    +\norm{\Psi_h(\Psi_h^{K}(x))-\Psi_h(X_*)}\right)^{1/2}+\norm{\beta_{h,{\boldsymbol{\omega}}}(X_*)}\\
&\leq\left(\norm{\alpha_{h,{\boldsymbol{\omega}}}(X_*)}^2+2C_0h\norm{\alpha_{h,{\boldsymbol{\omega}}}(X_*)}\norm{e_{K}}
    +(1-C_1h)\norm{e_{K}}^2\right)^{1/2}+\norm{\beta_{h,{\boldsymbol{\omega}}}(X_*)}\\
    &\leq\left(2\norm{\alpha_{h,{\boldsymbol{\omega}}}(X_*)}^2+(1-C_1h+C_0h^2)\norm{e_{K}}^2\right)^{1/2}+\norm{\beta_{h,{\boldsymbol{\omega}}}(X_*)}
    \end{aligned}
\end{gather*}
following the argument of \cite[Thm. 23]{sanz2021wasserstein}, using the contractive properties of $\Psi_h$ for strongly convex $g_{{\boldsymbol{\omega}}}$, and \cref{ass:InnerProd}. In the last line we use $2ab\leq a^2+b^2$. Then, via \cite[Lemma 28]{sanz2021wasserstein}, since
$$\norm{e_{K+1}}\leq \sqrt{(1-C)^2\norm{e_K}^2+2\norm{\alpha_{h,{\boldsymbol{\omega}}}(X_*)}^2}+\norm{\beta_{h,{\boldsymbol{\omega}}}(X_*)},$$ one has
\begin{gather*}
    \norm{e_K}\leq (1-C)^K\norm{e_0}+\sqrt{\frac{2}{C}\norm{\alpha_{h,{\boldsymbol{\omega}}}(X_*)}^2}+\frac{\norm{\beta_{h,{\boldsymbol{\omega}}}(X_*)}}{C}=\frac{2\norm{\alpha_{h,{\boldsymbol{\omega}}}(X_*)}}{\sqrt{C_1h+\mathcal{O}(h^2)}}+\frac{2\norm{\beta_{h,{\boldsymbol{\omega}}}(X_*)}}{C_1h+\mathcal{O}(h^2)},
\end{gather*}
since $\norm{e_0}=0$ and $C=C_1h/2+\mathcal{O}(h^2)$.

\section{Proofs of Convergence Guarantees}

In the following we will use $\|\cdot\|_{L^{2}} := (\mathbb{E}\|\cdot\|^{2})^{1/2}$. We also introduce the notation $\omega_{k}$ for $k \in \mathbb{N}$ to denote the vector $\boldsymbol{\omega}$ of batch indices at iteration $k$, which is defined by the procedure introduced in \cref{sec:randomisation_strategies}. 

\begin{proposition}\label{prop:aprioribounds}
Take $0<h< \min\{1/2R\gamma,1/2R\sqrt{L}\}$ and consider iterates $(x_{i},v_{i})_{i\in \mathbb{N}}$ of the stochastic gradient Polyak discretisation \cref{eq:Polyak} with the SMS or RR strategies. Assume \cref{assum:smoothness_sg} is satisfied. Assuming the stochastic gradients also satisfy \cref{assum:stochastic_gradient}, then we have for iterate $k \leq 2R$ that
\begin{align*}
    \|x_{k}-x_{0}\|_{L^{2}} &\leq 8hR \|v_{0}\|_{L^{2}} + 11h^{2}R^{2} L\|x_{0} - X_{*}\|_{L^{2}} + 11h^{2}R\sqrt{2R}\sigma_*.
\end{align*}
\end{proposition}
\begin{proof}
For $h < 1/(2\gamma R)$ and $k \leq 2R$ (where $\eta=e^{-\gamma h}$) one has
\begin{align*}
    &\|v_{k}\|_{L^{2}} \leq \eta^{k}\|v_{0}\|_{L^{2}} + h\|\sum^{k-1}_{i=0}\eta^{k-1-i}(\nabla f_{\omega_{i}}(x_{i})-\nabla f_{\omega_{i}}(X_{*}))\|_{L^{2}} + h\|\sum^{k-1}_{i=0}\eta^{k-1-i}\nabla f_{\omega_{i}}(X_{*})\|_{L^{2}} \\
    &\leq  \eta^{k}\|v_{0}\|_{L^{2}} + hL\sum^{k-1}_{i=0}\|x_{i}-X_{*}\|_{L^{2}} + h\sqrt{2R}\sigma_*
    \end{align*}
    using that $\eta<1$ and $\|\sum^{k-1}_{i=0}{\eta^{k-1-i}}\nabla f_{\omega_{i}}(X_{*})\|_{L^{2}}\leq(1+\eta^R)\sqrt{R/2}\sigma_*$ since via \cite[Lemma 1]{Mishchenko2020} one has $\|\sum^{j-1}_{i=0}{\eta^{j-1-i}}\nabla f_{\omega_{i}}(X_{*})\|_{L^{2}}\leq \sqrt{j(R-j)/(R-1)}\sigma_*$ for $j<R$. One may then derive
    \begin{align}
    &\|x_{k} - X_{*}\|_{L^{2}} \leq \|x_{0}-X_{*}\|_{L^{2}} + \eta h\sum^{k-1}_{i=0}\|v_{i}\|_{L^{2}} + h^{2}L\sum^{k-1}_{i=0}\|x_{i}-X_{*}\|_{L^{2}} + h^{2}\sqrt{2R}\sigma_*\nonumber\\
    &\leq \|x_{0} - X_{*}\|_{L^{2}} + 2hR\eta \|v_{0}\|_{L^{2}} + (2R\eta+1)h^{2} L \sum^{k-1}_{i=0}\|x_{i}-X_{*}\|_{L^{2}} + (2R\eta+1)h^{2}\sqrt{2R}\sigma_*\nonumber\\
    &\leq e^{5(hR)^{2}L }\left(\|x_{0} - X_{*}\|_{L^{2}} + 2hR \|v_{0}\|_{L^{2}}  + 3h^{2}R\sqrt{2R}\sigma_*\right).\label{eq:UsefulBound1}
\end{align}
Finally then,
\begin{align*}
\|x_{k}-x_{0}\|_{L^{2}} &\leq h\eta \sum_{i=0}^{k-1}\|v_{i}\|_{L^{2}} + h^{2} L \sum^{k-1}_{i=0}\|x_{i}-X_{*}\|_{L^{2}} +h^{2}R\sqrt{2R}\sigma_*\\
&\leq 2hR\eta\|v_{0}\|_{L^{2}} +(2R\eta+1) h^{2} L \sum^{k-1}_{i=0}\|x_{i}-X_{*}\|_{L^{2}} +(2R\eta+1)h^{2}\sqrt{2R}\sigma_*\\
    \|x_{k}-x_{0}\|_{L^{2}} &\leq 2hR \|v_{0}\|_{L^{2}} + 3h^{2}R^{2} L e^{5(hR)^{2}L }\left(\|x_{0} - X_{*}\|_{L^{2}} + 2hR \|v_{0}\|_{L^{2}}  + 3h^{2}R\sqrt{2R}\sigma_*\right) \\
    &+ 3h^{2}R\sqrt{2R}\sigma_*,
\end{align*}
and if we assume that $h<\frac{1}{2R\sqrt{L}}$ we have 
\begin{align*}
    \|x_{k}-x_{0}\|_{L^{2}} &\leq 8hR \|v_{0}\|_{L^{2}} + 11h^{2}R^{2} L\|x_{0} - X_{*}\|_{L^{2}} + 11h^{2}R\sqrt{2R}\sigma_*.
\end{align*}
\end{proof}

\begin{lemma}\label{lemma:SGD}
Consider the stochastic gradient Polyak discretisation \cref{eq:Polyak} with the SMS or RR strategies such that $\gamma > 0$, $0<h< \min\{1/2R\gamma,1/2R\sqrt{L}\}$ and \cref{assum:smoothness_sg} is satisfied. Assuming the stochastic gradients satisfy \cref{assum:stochastic_gradient}, then the iterates $(x_{i},v_{i})_{i\in \mathbb{N}}$ for $k \in \mathbb{N}$ satisfy
    \begin{align*}
    \mathbb{E}\mathcal{V}(x_{2kR},v_{2kR}) \leq (1-chR)^{k}\mathcal{V}(x_{0},v_{0}) + \frac{C(\mu,L,\gamma)\left[h^{5}R^{4}\sigma_*^{2} + \|J_{v}\|^{2}_{L^{2}}\right]}{chR},
\end{align*}
where $J_{v} := h\sum^{2R-1}_{i=0}\eta^{2R-1-i}\nabla f_{\omega_{i}}(X_{*})$ and
\begin{align*}
    c&= \frac{\min\left\{\gamma,\mu\gamma/\gamma^{2}_{h}\right\}}{4} - C(\mu,L,\gamma)hR.
\end{align*}
\end{lemma}
\begin{proof}
We consider the Lyapunov function $\mathcal{V}:\mathbb{R}^{d} \times \mathbb{R}^{d} \to \mathbb{R}$, defined for $(x,v)\in \mathbb{R}^{2d}$ by \eqref{eq:Lyapunov}.
Using \cite[Lemma 1.2.3]{NesterovBook} we have that
\[
F(x_{2R}) - F(X_{*}) \leq F(x_{0}) - F(X_{*}) + \langle\nabla F(x_{0}), x_{2R}-x_{0}\rangle + \frac{L}{2}\|x_{2R}-x_{0}\|^{2},
\]
and thus
\begin{align}
    \mathcal{V}(x_{2R},v_{2R}) &= F(x_{2R}) - F(X_{*}) + \frac{\gamma^{2}_{h}}{4}\|x_{2R}-X_{*}\|^{2} + \frac{\gamma_{h}}{2}\left\langle x_{2R}-X_{*},v_{2R}\right\rangle + \frac{1}{2}\|v_{2R}\|^{2}\nonumber\\
    &\leq F(x_{0}) - F(X_{*}) + \langle\nabla F(x_{0}), x_{2R}-x_{0}\rangle + \frac{L}{2}\|x_{2R}-x_{0}\|^{2} + \frac{\gamma^{2}_{h}}{4}\|x_{2R}-X_{*}\|^{2}\nonumber \\
    &+ \frac{\gamma_{h}}{2}\left\langle x_{2R}-X_{*},v_{2R}\right\rangle + \frac{1}{2}\|v_{2R}\|^{2}\label{eq:LLyapunovBound1}.
\end{align}
Then for iterates generated via the Heavy Ball method and setting $\gamma_{h} = \frac{1-\eta}{h\eta}$, one has
\begin{align*}
    &v_{2R} = \eta^{2R}v_{0} - \underbrace{h\sum^{2R-1}_{i=0}\eta^{2R-1-i}\nabla f_{\omega_{i}}(x_{i})}_{:= D_{v}}.
\end{align*}
Similarly, one may derive that $x_{k+1}+\gamma_h^{-1}v_{k+1}=x_{k}+\gamma_h^{-1}v_{k}-(h^2+h\gamma_h^{-1})\nabla f_{\omega_k}(x_k)$ and so then
\begin{align*}
    &x_{2R} = -\gamma^{-1}_{h}v_{2R}+ x_{0} + \gamma^{-1}_{h}v_{0} - (h^2+h\gamma^{-1}_{h})\sum_{k=0}^{2R-1}\nabla f_{\omega_k}(x_k)\\
    &= x_{0} + \gamma^{-1}_{h}(1 - \eta^{2R})v_{0} - \underbrace{\left[h^{2}\sum^{2R-1}_{i=0} \nabla f_{\omega_{i}}(x_{i}) + h\gamma^{-1}_{h}\sum^{2R-1}_{i=0}(1-\eta^{2R-1-i})\nabla f_{\omega_{i}}(x_{i})\right]}_{:= D_{x}}.
\end{align*}

Substituting these expressions for $x_{2R},v_{2R}$ into \cref{eq:LLyapunovBound1}, we arrive at the following bound on the expectation of  $\mathcal{V}(x_{2R},v_{2R})$, where the first order terms (in terms of the stepsize $h$) are given in the same lines as $\mathcal{V}(x_{0},v_{0})$ (we suppress the $\mathbb{E}$ notation, but it is implicit in every expression)
\begin{align*}
    &\mathcal{V}(x_{2R},v_{2R}) \leq \mathcal{V}(x_{0},v_{0}) + {\langle\nabla F(x_{0}), \gamma^{-1}_{h}(1-\eta^{2R})v_{0}\rangle - \frac{\gamma_{h}}{2}\left\langle x_{0}-X_{*}, D_{v}\right\rangle - \frac{1}{2}(1-\eta^{2R})\|v_{0}\|^{2} -\langle v_{0}, D_{v}}\rangle\\
    &- \langle\nabla F(x_{0}), D_{x}\rangle + \left(\frac{\gamma^{2}_{h}}{4} + \frac{L}{2}\right)\|\gamma^{-1}_{h}(1-\eta^{2R})v_{0}-D_{x}\|^{2} - \frac{\gamma^{2}_{h}}{2}\left\langle x_{0}-X_{*},D_{x}\right\rangle\\
        & - \frac{\gamma_{h}}{2}\left\langle D_{x}, \eta^{2R}v_{0} - D_{v}\right\rangle  +\frac{1}{2}{(1-\eta^{2R})\langle v_{0}, D_{v}}\rangle + \frac{1}{2}\|D_{v}\|^{2}.
\end{align*}
We may rewrite $h^{-1}D_{v} = \sum^{2R-1}_{i=0}\eta^{2R-1-i}\nabla f_{\omega_{i}}(x_{i}) = \sum^{2R-1}_{i=0}\eta^{2R-1-i}\nabla f_{\omega_{i}}(x_{0}) + \sum^{2R-1}_{i=0}\eta^{2R-1-i}(\nabla f_{\omega_{i}}(x_{i}) - \nabla f_{\omega_{i}}(x_{0}))$ (which is at least second order in the stepsize) and then use that $\mathbb{E}_{x_0}\left[\sum^{2R-1}_{i=0}\eta^{2R-1-i}\nabla f_{\omega_{i}}(x_{0})\right]=\frac{(1-\eta^{2R})}{1-\eta}\nabla F(x_0)$ to rewrite the above bound as (recall we suppress the $\mathbb{E}$ symbol)
\begin{align*}
    \mathcal{V}(x_{2R},v_{2R})&\leq \mathcal{V}(x_{0},v_{0})  { - \frac{(1-\eta^{2R})}{2\eta}\left\langle x_{0}-X_{*}, \nabla F(x_{0})\right\rangle - \frac{1}{2}(1-\eta^{2R})\|v_{0}\|^{2}}\\
    &{- \frac{\gamma_{h}}{2}h\left\langle x_{0}-X_{*},\sum^{2R-1}_{i=0}\eta^{2R-1-i}(\nabla f_{\omega_{i}}(x_{i}) - \nabla f_{\omega_{i}}(x_{0}))\right\rangle} \\
    &-h(1-\eta^{2R})\langle\nabla F(x_{0}) , v_{0}\rangle- h \langle v_{0} , \sum^{2R-1}_{i=0}\eta^{2R-1-i}(\nabla f_{\omega_{i}}(x_{i}) - \nabla f_{\omega_{i}}(x_{0}))\rangle \\
    &- \langle\nabla F(x_{0}), D_{x}\rangle + \left(\frac{\gamma^{2}_{h}}{4} + \frac{L}{2}\right)\|\gamma^{-1}_{h}(1-\eta^{2R})v_{0}-D_{x}\|^{2} - \frac{\gamma^{2}_{h}}{2}\left\langle x_{0}-X_{*},D_{x}\right\rangle\\
    & - \frac{\gamma_{h}}{2}\left\langle D_{x}, \eta^{2R}v_{0} - D_{v}\right\rangle +\frac{1}{2} (1-\eta^{2R})\langle v_{0}, D_{v} \rangle+ \frac{1}{2}\|D_{v}\|^{2}.
\end{align*}
We then use strong convexity of $F$ and the fact that $1/\eta>1$ to write
\begin{align*}
    \mathcal{V}(x_{2R},v_{2R})&\leq \mathcal{V}(x_{0},v_{0})  - \frac{(1-\eta^{2R})}{2}\left(\frac{\mu}{2}\|x_{0} -X_{*}\|^{2} + F(x_{0}) - F(X_{*}) + \|v_{0}\|^{2}\right)\\
    &\underbrace{- \frac{\gamma_{h}}{2}h\left\langle x_{0}-X_{*} + 2\gamma^{-1}_{h}v_{0},\sum^{2R-1}_{i=0}\eta^{2R-1-i}(\nabla f_{\omega_{i}}(x_{i}) - \nabla f_{\omega_{i}}(x_{0}))\right\rangle}_{(I)}\\
    &+\underbrace{h(1-\eta^{2R})\langle\nabla F(x_{0}) , v_{0}\rangle}_{(II)} + \underbrace{\left(\frac{\gamma^{2}_{h}}{2} + L\right)\left(\|\gamma^{-1}_{h}(1-\eta^{2R})v_{0}\|^{2}+{\|D_{x}\|^{2}}\right)}_{(III)} \\
    &\underbrace{- \frac{\gamma^{2}_{h}}{2}\left\langle x_{0}-X_{*} + 2\gamma^{-2}_{h}\nabla F(x_{0}) +\gamma^{-1}_{h}{\eta^{2R}}v_{0},D_{x}\right\rangle}_{(IV)}\\
    & +\underbrace{{\frac{\gamma_{h}}{2}\left\langle D_{x}, D_{v}\right\rangle}}_{(V)} + \underbrace{\frac{1}{2}(1-\eta^{2R})\langle v_{0}, D_{v}\rangle}_{(VI)} + \underbrace{\frac{1}{2}{\|D_{v}\|^{2}}}_{(VII)}.
\end{align*}
We can bound each term separately in expectation. Firstly, via the Cauchy-Schwartz inequality then the Peter-Paul inequality
\begin{align*}
    &\mathbb{E}[(I)] \leq \frac{1}{32}hr\left(\gamma^{2}_{h}\|x_{0}-X_{*}\|^{2} + 2\|v_{0}\|^{2}\right) + \frac{32h}{r}\|\sum^{R-1}_{i=0}\eta^{R-1-i}(\nabla f_{\omega_{R-1-i}}(x_{R+i}) - \nabla f_{\omega_{R-1-i}}(x_{0})) \\
    &+ \eta^{R+i}(\nabla f_{\omega_{R-1-i}}(x_{R-1-i})-\nabla f_{\omega_{R-1-i}}(x_{0}))\|^{2}\\
    &\leq \frac{1}{2}hr\mathcal{V}(x_{0},v_{0}) +\frac{C(hR)^{3}}{r/R}L \left((1 + hR L/\gamma_{h})^{2}\mathcal{V}(x_{0},v_{0}) + h^{2}R\sigma_*^{2}\right),
\end{align*}
where we finish by applying \cref{prop:aprioribounds} and the first bound in \cref{lemma:equiv}. By the same bound one has for $(II)$ and then $(III)$
\begin{align*}
    (II) &\leq \frac{\sqrt{L}}{2}h(1-\eta^{2R})\|v_{0}\|^{2} + \frac{1}{2\sqrt{L}}h(1-\eta^{2R})\|\nabla F(x_{0})\|^{2}\\
    &\leq 6h(1-\eta^{2R})\max{\left\{\frac{L^{3/2}}{\gamma^{2}_{h}},\sqrt{L}\right\}}\mathcal{V}(x_{0},v_{0})\\
    (III) &\leq \left(\frac{\gamma^{2}_{h}}{2} + L\right)\left(8\gamma^{-2}_{h}(1-\eta^{2R})^{2}\mathcal{V}(x_{0},v_{0})+\|D_{x}\|^{2}\right),
\end{align*}
where we apply the Peter-Paul inequality for $(II)$ and the triangle inequality for $(III)$.
For $(IV)$ we first take the expectation inside the inner product (conditional on $(x_0,v_0)$) and then apply the Peter-Paul inequality (we also use that $\eta^{4R}\leq e^{-2}$).
\begin{align*}
    (IV) &\leq \frac{\gamma^{4}_{h}(hR)^{2}}{4}\|x_{0}-X_{*} + 2\gamma^{-2}_{h}\nabla F(x_{0}) +\gamma^{-1}_{h}\eta^{2R}v_{0}\|^{2} + \frac{\|\mathbb{E}[D_{x}]\|^{2}}{2(hR)^{2}}\\
    &\leq 6(hR)^{2}\gamma^{2}_{h}(1+ 4\gamma^{-4}_{h}L^{2})\mathcal{V}(x_{0},v_{0})+ \frac{\|\mathbb{E}[D_{x}]\|^{2}}{2(hR)^{2}}.
\end{align*}
The remaining bounds for $(V)$ and $(VI)$ are obtained rather immediately,
\begin{align*}
    (V) &\leq \frac{\gamma^{2}_{h}}{4}\|D_{x}\|^{2} + \frac{1}{2}\|D_{v}\|^{2}\\
    (VI) &\leq \frac{1}{4}(1-\eta^{2R})^{2} \|v_{0}\|^{2} + \frac{1}{4}\|D_{v}\|^{2} \leq 2(1-\eta^{2R})^{2}\mathcal{V}(x_{0},v_{0})+ \frac{1}{4}\|D_{v}\|^{2}.
\end{align*}
Now we wish to bound the stochastic gradient terms $D_v$ and $D_x$. Firstly we have that 
\begin{align*}
    D_{v} = h\sum^{2R-1}_{i=0}\eta^{2R-1-i}\left(\nabla f_{\omega_{i}}(x_{i}) - \nabla f_{\omega_{i}}(X_{*})\right) + h\sum^{2R-1}_{i=0}\eta^{2R-1-i}\nabla f_{\omega_{i}}(X_{*}),
\end{align*}
which may thus be bounded in $L^{2}$ using \cref{eq:UsefulBound1} as
\begin{align*}
    \|D_v\|_{L^2}&\leq C\left(hRL\left((\gamma^{-1}_{h} + hR)\mathcal{V}^{1/2}(x_{0},v_{0})  + h^{2}R\sqrt{R}\sigma_* \right) + \bigg\|\underbrace{h\sum^{2R-1}_{i=0}\eta^{2R-1-i}\nabla f_{\omega_{i}}(X_{*})}_{:= J_{v}}\bigg\|_{L^{2}}\right).
\end{align*}
Further considering $D_{x}$ we have
\begin{align*}
    D_{x} = \sum^{2R-1}_{i=0} (h^{2} + h\gamma^{-1}_{h}(1-\eta^{2R-1-i}))\left(\nabla f_{\omega_{i}}(x_{i}) - \nabla f_{\omega_{i}}(X_{*})\right) + \sum^{2R-1}_{i=0} (h^{2} + h\gamma^{-1}_{h}(1-\eta^{2R-1-i}))\nabla f_{\omega_{i}}(X_{*}),
\end{align*}
with the final term vanishing in expectation. Then as with $D_{v}$ we can bound $D_{x}$ in $L^{2}$ by
\begin{align*}
    &\leq C(h^{2} + h\gamma^{-1}_{h}(1-\eta^{2R-1}))RL\left((\gamma^{-1}_{h} + hR)\mathcal{V}^{1/2}(x_{0},v_{0})  + h^{2}R\sqrt{R}\sigma_* \right) \\
    &+ \bigg\|\underbrace{\sum^{2R-1}_{i=0} (h^{2} + h\gamma^{-1}_{h}(1-\eta^{2R-1-i}))\nabla f_{\omega_{i}}(X_{*})}_{:= J_{x} = \gamma^{-1}_{h}J_{v} \textnormal{ (using that }\sum^{2R-1}_{i=0}\nabla f_{\omega_{i}}(X_{*}) = 0\textnormal{)}}\bigg\|_{L^{2}}.
\end{align*}
Combining terms (where we use the second bound in \cref{lemma:equiv}) and considering $r = \frac{\gamma R\min\left\{1,\mu/\gamma^{2}_{h}\right\}}{2}$, we use that $h<\frac{1}{2\gamma R}$ and $-(1-\eta^{2R})\leq -R\gamma h$ to obtain 
\begin{align*}
    &\mathcal{V}(x_{2R},v_{2R}) \leq \mathcal{V}(x_{0},v_{0})  - \frac{(1-\eta^{2R})}{2}\min\left\{1,\mu/\gamma^{2}_{h}\right\}\mathcal{V}(x_{0},v_{0})\\
    &+ \left[\frac{h}{2}r + C(\mu,L,\gamma)(hR)^{2}\right]\mathcal{V}(x_{0},v_{0}) + C(\mu,L,\gamma)\left[h^{5}R^{4}\sigma_*^{2} + \|J_{v}\|^{2}_{L^{2}}\right]\\
    &\leq \underbrace{\left(1-\frac{h R\min\left\{\gamma,\mu\gamma/\gamma^{2}_{h}\right\}}{4} + C(\mu,L,\gamma)(hR)^{2}\right)}_{:= 1- chR}\mathcal{V}(x_{0},v_{0}) + C(\mu,L,\gamma)\left[h^{5}R^{4}\sigma_*^{2} + \|J_{v}\|^{2}_{L^{2}}\right],
\end{align*}
We end up with 
\begin{align*}
    \mathcal{V}(x_{2kR},v_{2kR}) \leq (1-chR)^{k}\mathcal{V}(x_{0},v_{0}) + \frac{C(\mu,L,\gamma)\left[h^{5}R^{4}\sigma_*^{2} + \|J_{v}\|^{2}_{L^{2}}\right]}{chR},
\end{align*}
as required.
\end{proof}

\begin{lemma}\label{lem:randomization}
For the RR randomisation strategy we have
\begin{align*}
    \left\|h\sum^{2R-1}_{i=0}\eta^{2R-1-i}\nabla f_{\omega_{i}}(X_{*})\right\|_{L^{2}} &\leq 4h^{2}\gamma R^{3/2}\sigma_*.
\end{align*}
For the SMS randomisation strategy we have
\begin{align*}
    \left\|h\sum^{2R-1}_{i=0}\eta^{2R-1-i}\nabla f_{\omega_{i}}(X_{*})\right\|_{L^{2}} &\leq 2h^{3}\gamma^{2}R^{5/2}\sigma_*.
\end{align*}
\end{lemma}
\begin{proof}
Firstly, for the RR randomisation strategy we bound (using that $(\eta^{2R}-1)^2=(e^{-2R\gamma h}-1)^2\leq (2R\gamma h)^2$)
\begin{align*}
    \left\|h \sum^{2R-1}_{i=0}\eta^{2R-1-i}\nabla f_{\omega_{i}}(X_{*})\right\|_{L^{2}} &=  \left\|h \sum^{2R-1}_{i=0}(\eta^{2R-1-i}-1)\nabla f_{\omega_{i}}(X_{*})\right\|_{L^{2}}\\
    &\leq 2h\sqrt{\sum^{2R-1}_{i=0}(\eta^{2R-1-i}-1)^{2}\|\nabla f_{\omega_{i}}(X_{*})\|^{2}_{L^{2}}}\\
    &\leq 4h^{2}\gamma R\sqrt{R} \sigma_*,
\end{align*}
where we have used \cite[Lemma 10]{paulin2024sampling} (since $\eta^{2R-1-i}$ introduces a dependence on the index $i$ which prevents the use of e.g. \cite[Lemma 1]{Mishchenko2020}). For the SMS randomisation strategy we bound 
\begin{align*}
    \left\|h \sum^{2R-1}_{i=0}\eta^{2R-1-i}\nabla f_{\omega_{i}}(X_{*})\right\|_{L^{2}} &=  h\left\|\sum^{R-1}_{i=0}(\eta^{R-i-1}-2\eta^{(2R-1)/2} + \eta^{R+i} )\nabla f_{\omega_{R-i-1}}(X_{*})\right\|_{L^{2}}\\
    &= h\left\|\sum^{R-1}_{i=0}\eta^{R-i-1}(1-\eta^{i+1/2})^{2}\nabla f_{\omega_{R-i-1}}(X_{*})\right\|_{L^{2}}\\
    &\leq 2h\sqrt{\sum^{R-1}_{i=0}\eta^{2(R-i-1)}(1-\eta^{i+1/2})^{4}\|\nabla f_{\omega_{R-i-1}}(X_{*})\|^{2}_{L^{2}}}\\
    &\leq 2h(h\gamma R)^{2}\sqrt{R}\sigma_*,
\end{align*}
as required.
\end{proof}
\end{document}